\documentclass[a4paper,12pt,reqno]{amsart}
\usepackage{amsmath,amsfonts,amssymb,amsthm,enumerate,multicol}
\usepackage{tikz,graphicx}
\usepackage{hyperref}
\usepackage{caption,enumitem,wasysym}
\usepackage{cleveref}
\usepackage{epstopdf}
\usepackage{float,wrapfig}
\usepackage[labelsep=space]{caption}
\usepackage[font=footnotesize]{caption}
\usepackage{xcolor}

\newtheorem{theorem}{Theorem}[section]
\newtheorem{definition}{Definition}[section]
\newtheorem{example}{Example}[section]
\newtheorem{corollary}{Corollary}[section]
\newtheorem{remark}{Remark}[section]

\numberwithin{equation}{section}

\title{Chain sequences and Zeros of a perturbed $R_{II}$ type recurrence relation}
\author{Vinay Shukla$^\dagger$}
\address{$^\dagger$Department of Mathematics\\ Indian Institute of Technology, Roorkee-247667, Uttarakhand, India}
\email{vshukla@ma.iitr.ac.in}
\author[A. Swaminathan]{Anbhu Swaminathan\, $^{\#\ddagger}$}{\thanks{$^{\#}$Corresponding author}}
   \address{$^\ddagger$Department of Mathematics\\ Indian Institute of Technology, Roorkee-247667, Uttarakhand, India}
   \email{mathswami@gmail.com, a.swaminathan@ma.iitr.ac.in}
\allowdisplaybreaks
\bigskip

\begin{document}

	
\begin{abstract}
In this manuscript, new algebraic and analytic aspects of the orthogonal polynomials satisfying $R_{II}$ type recurrence relation given by 
\begin{align*}
\mathcal{P}_{n+1}(x) = (x-c_n)\mathcal{P}_n(x)-\lambda_n (x-a_n)(x-b_n)\mathcal{P}_{n-1}(x), \quad n \geq 0,
\end{align*}
where $\lambda_n$ is a positive chain sequence and $a_n$, $b_n$, $c_n$ are sequences of real or complex numbers with $\mathcal{P}_{-1}(x) = 0$ and $\mathcal{P}_0(x) = 1$ are investigated when the recurrence coefficients are perturbed. Specifically, representation of new perturbed polynomials (co-polynomials of $R_{II}$ type) in terms of original ones with the interlacing and monotonicity properties of zeros are given. For finite perturbations, a transfer matrix approach is used to obtain new structural relations. Effect of co-dilation in the corresponding chain sequences and their consequences onto the unit circle are analysed. A particular perturbation in the corresponding chain sequence called complementary chain sequences and its effect on the corresponding Verblunsky coefficients is also studied. 
\end{abstract}

\maketitle
\markboth{Vinay Shukla and A. Swaminathan}{Perturbed $R_{II}$ type recurrence relation}

\section{Introduction}\label{Introduction}
The fact that the construction of new sequences by modifying the original sequence is a powerful tool with many applications to the theoretical and physical problems form the basis of perturbation theory.
Orthogonal polynomials on the real line (OPRL) satisfy a second order recurrence relation \cite{Esmail book} of the form
\begin{align}\label{OPRL TTRR R2}
	\mathcal{R}_{n+1}(x) = (x-b_{n})\mathcal{R}_n(x)-\gamma_n \mathcal{R}_{n-1}(x), \quad \mathcal{R}_{-1}(x) = 0, \quad \mathcal{R}_0(x) = 1, \quad n \geq 0.
\end{align}
The analysis by adding a constant to the first coefficient $b_0$, called co-recursive, is available in the pioneer work \cite{chihara PAMS 1957}. Such perturbations are not artificial as they are useful objects in the study of Hamiltonian operator appearing in quantum mechanics. Motivated by this problem, properties of the co-recursive, co-dilated and co-modified polynomials have been studied \cite{Dini thesis 1988, Ronveaux vigo 1988, Slim 1988}. Later, it was needed to make perturbations anywhere on the surface of the target many-body system and to study how the spectroscopic properties gets changed. This led to the study of generalized co-recursive, co-dilated and co-modified polynomials \cite{Paco perturbed recurrence 1990}. In this direction, finite perturbations were studied in \cite{Peherstorfer 1992}. For perturbations of recurrence coefficients occuring in higher order recurrences and its extensions to Sobolev OPRL, see \cite{Leopold 2 2007, Leopold 3 2008}. Recently, a transfer matrix approach is introduced in \cite{Castillo co-polynomials on real line 2015} to study polynomials perturbed in a (generalized) co-dilated and/or co-recursive way. \par 
The sequence of monic orthogonal polynomials on the unit circle (OPUC) denoted by $\{\phi_n(z)\}_{n=0}^\infty$, satisfy Szeg\H{o} recurrence given as
\begin{align}\label{Szego recurrence R2}
	\begin{bmatrix}
		\phi_{n+1}(z)	\\
		\phi^*_{n+1}(z)
	\end{bmatrix} & = T_n(z) \begin{bmatrix}
		\phi_{n}(z)	\\
		\phi^*_{n}(z)
	\end{bmatrix}, \quad T_n(z) = \begin{bmatrix}
		z &	-\overline{\alpha_n} \\ 
		-\alpha_n z & 1
	\end{bmatrix},
\end{align}
with initial condition $\phi_0(z)=1$, where $\phi^*_{n}(z)=z^n \overline{\phi_n(1/z)}$ is the reversed polynomial and $T_n(z)$ is said to be the transfer matrix \cite{Simon part1 2005}. The elements of the sequence $\{\alpha_n\}_{n=0}^\infty$ where $\alpha_n = -\overline{\phi_{n+1}(0)}$ are known as Verblunsky coefficients and lie in the unit disc. \par 
Verblunsky theorem states that the OPUC are completely determined by their reflection coefficients. This fact motivated the authors in \cite{Castillo perturbed szego 2014} to study perturbations of Verblunsky coefficients. Again, a transfer matrix approach is used to study the so called co-polynomials on unit circle (COPUC). The structural relations and rational spectral transformation for C-functions, associated with COPUC have been discussed.  \par 
The $R_I$ type recurrence relation \cite{Esmail masson JAT 1995} are given as
\begin{align}\label{R1 R2 pp}
	&  \mathcal{U}_{n+1}(z) = k_n(z-c_n)\mathcal{U}_n(z)-\lambda_n (z-a_n)\mathcal{U}_{n-1}(z), \quad n\geq 0, \\ 
	& \nonumber	\mathcal{U}_{-1}(z) = 0, \qquad \mathcal{U}_0(z) = 1. 
\end{align}
Perturbations of parameters of $R_{I}$ type recurrence relation, distribution of zeros, interlacing properties and spectral transformations have been dealt in detail in a previous work by the authors \cite{swami vinay 2021}. \par 
Recurrence relation of the form
\begin{align}\label{R2}
&  \mathcal{P}_{n+1}(x) = (x-c_n)\mathcal{P}_n(x)-\lambda_n (x-a_n)(x-b_n)\mathcal{P}_{n-1}(x), \quad n\geq 0, \\ 
& \nonumber	\mathcal{P}_{-1}(x) = 0, \qquad \mathcal{P}_0(x) = 1. 
\end{align}
has been studied in \cite{Esmail masson JAT 1995}. Infact, it was shown that if $\lambda_n \neq 0$ with $\mathcal{P}_n(a_n)\neq 0$,  $\mathcal{P}_n(b_n)\neq 0$ for $n\geq 0$, then there exits a rational function $\psi_n(x)=\dfrac{\mathcal{P}_n(x)}{\prod_{j=1}^{n}(x-a_j)(x-b_j)}$ and a linear functional $\mathfrak{N}$ such that the orthogonality relations
\begin{align*}
	\mathfrak{N}\left[x^k \psi_n(x) \right]\neq  0, \quad 0\leq k < n,
\end{align*}
holds \cite[Theorem 3.5]{Esmail masson JAT 1995}. Conversely, starting from a sequence of rational function $\psi_n(x)$ satisfying a three-term recurrence relation and having poles at $\{a_k\}_{k=0}^\infty$ and $\{b_k\}_{k=0}^\infty$, we can obtaion \eqref{R2}. Following \cite{Esmail masson JAT 1995}, the recurrence relation \eqref{R2} will be referred as recurrence relation of $R_{II}$ type and the $\mathcal{P}_n(x), n\geq 1$, generated by it as $R_{II}$ polynomials. \par 
Note that the infinite continued fraction
\begin{align}\label{R2 continued fraction}
	\mathcal{R}_{II}(x) = \frac{1}{x-c_0} \mathbin{\genfrac{}{}{0pt}{}{}{-}} \frac{\lambda_1(x-a_1)(x-b_1)}{x-c_1} \mathbin{\genfrac{}{}{0pt}{}{}{-}} \frac{\lambda_2(x-a_2)(x-b_2)}{x-c_2} \mathbin{\genfrac{}{}{0pt}{}{}{-}} \mathbin{\genfrac{}{}{0pt}{}{}{\cdots}} ,
\end{align}
terminates when $x=a_k$ or $x=b_k$, $k \geq 1$. Following \cite{Esmail masson JAT 1995}, we call it as $R_{II}$-fraction. 
For the ease of notation (following \cite{Beckermann Derevyagin Zhedanov linear pencil 2010}), \eqref{R2 continued fraction} can also be written as
\begin{align}
\mathcal{R}_{II}(x) = \frac{1}{w_{0}(x)} \mathbin{\genfrac{}{}{0pt}{}{}{-}} \frac{\chi_{1}^{L}(x)\chi_{1}^{R}(x)}{w_{1}(x)} \mathbin{\genfrac{}{}{0pt}{}{}{-}} \frac{\chi_{2}^{L}(x)\chi_{2}^{R}(x)}{w_{2}(x)} \mathbin{\genfrac{}{}{0pt}{}{}{-}} \mathbin{\genfrac{}{}{0pt}{}{}{\cdots}} ,
\end{align}
where $w_{j}(x)$, $\chi_{j}^{L}(x)$ and $\chi_{j}^{R}(x)$ are non-zero polynomials of degree one. 
The polynomials $\mathcal{P}_n(x)$ of degree atmost $n$ are the numerator polynomials associated with \eqref{R2 continued fraction}. Also, the polynomials of second kind of degree atmost $n-1$ \cite[page 16]{Esmail masson JAT 1995} associated with recurrence relation \eqref{R2} are given by
\begin{align*}
	&  \mathcal{Q}_{n+1}(x) = (x-c_n)\mathcal{Q}_n(x)-\lambda_n (x-a_n)(x-b_n)\mathcal{Q}_{n-1}(x), \quad n\geq 1, \\ 
	& \nonumber	\mathcal{Q}_{0}(x) = 0, \qquad \mathcal{Q}_1(x) = 1,
\end{align*}
such that the rational function $\dfrac{\mathcal{Q}_n(x)}{\mathcal{P}_n(x)}$ is the $n$-th convergent of the continued fraction \eqref{R2 continued fraction} and is associated to a linear pencil $x\mathcal{J}_n-\mathcal{H}_n$ where $\mathcal{J}_n$ and $\mathcal{H}_n$ are tridiagonal. In matrix notation, we have,
\begin{align*}
	x\mathcal{J}_n-\mathcal{H}_n = \begin{pmatrix}
		w_{0}(x) &	-\chi_{1}^{R}(x) & 0 & \ldots & 0 & 0\\ 
		-\chi_{1}^{L}(x) &  w_{1}(x) & -\chi_{2}^{R}(x) & \ldots & 0 & 0\\
		0 & -\chi_{2}^{L}(x) & w_{2}(x) &  \ldots & 0 & 0\\
	\vdots	& \vdots & \vdots & \ddots & \vdots & \vdots\\
	0 & 0 & 0 & \ldots & w_{n-1}(x) & -\chi_{n}^{R}(x)\\
	0 & 0 & 0 & \ldots & -\chi_{n}^{L}(x) & w_{n}(x)
	\end{pmatrix}.
\end{align*}
The polynomial $\mathcal{P}_n(x)$ satisfying \eqref{R2} is the characteristic polynomial of the pencil matrix $x\mathcal{J}_n-\mathcal{H}_n$ \cite{Zhedanov Rational spectral 1997}. In \cite[Theorem 3.7]{Esmail masson JAT 1995}, the existence of a natural Borel measure associated with $R_{II}$-fraction was also established. $R_{II}$ type recurrence relation and corresponding $R_{II}$ polynomials have a rich literature as they are dealt in numerous ways by many authors. The little and big $-1$ Jacobi polynomials and their $q$ analogues are extensively studied in the literature \cite{ Derevyagin Vinet Zhedanov Constr. Approx. 2012, Esmail book, Vinet Zhedanov transaction 2012}. It is illustrated in \cite{KKB Swami PAMS 2019} that they can also be related to $R_{II}$- type recurrence relation either via linear pencil matrix or by a simple m\"obius transformation. The rational functions satisfying a doubly spectral relation which can be reduced to $R_{II}$- type recurrence relation under specific conditions have been shown related to Pseudo-Jacobi polynomials, which are also called the Routh-Romanovski polynomials \cite{Derevyagin JAT 2017}. Further, the complementary Routh-Romanovski polynomials (abbreviated as CRR polynomials), which follow from complex form of Jacobi polynomials, satisfy a special $R_{II}$- type recurrence \eqref{special R2}. Such polynomials are shown to be related to a special class of orthogonal polynomials on the unit circle. Also, these are useful in the study of Schr\H{o}dinger equations, regular Coulomb wave functions and extended regular Coulomb wave functions \cite{Finkelshtein Ribeiro Ranga Tyaglov PAMS 2019, Finkelshtein Ribeiro Ranga Tyaglov CRR 2020}. Interested authors may look at \cite{KKB Swami PAMS 2019, Bracciali Pereira ranga 2020, Esmail Ranga 2018} and references therein for some recent progress related to $R_{II}$ type recurrence relations and $R_{II}$ polynomials. \par 
The objective of this manuscript is to study some properties of polynomials which satisfy the recurrence relation \eqref{R2} with new recurrence coefficients, perturbed in a (generalized) co-dilated and/or co-recursive way, i.e.,
\begin{align}\label{R2 perturbed}
	\mathcal{P}_{n+1}(x;\mu_k,\nu_k) = (x-c^*_n)\mathcal{P}_n(x;\mu_k,\nu_k)-\tilde{\lambda}_n (x-a_n)(x-b_n)\mathcal{P}_{n-1}(x;\mu_k,\nu_k),
\end{align}
with initial conditions $\mathcal{P}_0(x) = 1$ and $\mathcal{P}_{-1}(x) =0$. In other words, we consider arbitrary single modification of recurrence coefficients as follows:
\begin{align}
	c^*_n &= c_n+\mu_k \delta_{n,k}, \qquad \rm(co-recursive) \label{co-recursive condition R2}\\
	\tilde{\lambda}_n &= \nu_k^{\delta_{n,k}} \lambda_n, \qquad \rm(co-dilated) \label{co-dilated condition R2}
\end{align}
where $k$ is a fixed non-negative integer number. As far as we know, this problem has not been studied in the literature. \par 
A sequence $\{\lambda_{n+1}\}_{n \geq 1}$ is a positive chain sequence \cite{Chihara book 1978}, if there exists another sequence $\{m_n\}_{n \geq 1}$ such that $0 \leq m_1 <1$, $ 0 <m_{n+1} < 1$ and $(1-m_{n})m_{n+1} = \lambda_{n+1}$ for $n\geq 1$. The sequence $\{m_n\}_{n \geq 1}$ is called a parameter sequence of the positive chain sequence $\{\lambda_{n+1}\}_{n \geq 1}$. A positive chain sequence can have multiple (infinitely many) parameter sequences, but it always has a minimal parameter sequence. Denote by $\{l_n\}_{n \geq 1}$, the minimal parameter sequence of the positive chain sequence $\{\lambda_{n+1}\}_{n \geq 1}$ which is given by $l_1 = 0$, $ 0 <l_{n+1} < 1$ and $(1-l_{n})l_{n+1} = \lambda_{n+1}$ for $n\geq 1$. We study the effect of a special type of perturbation called co-dilation in the chain sequence $\{\lambda_{n+1}\}_{n \geq 1}$, that gives a new set of Verblunsky coefficients from which the sequence of Szeg\H{o} polynomials is constructed (see \Cref{Co-polynomials and consequences on unit circle} and \Cref{Co-dilation and Chain sequences}). A new sequence of $R_{II}$ polynomials is obtained in two ways: one, in \Cref{Connection with unit circle}, via a direct computation using Szeg\H{o} polynomials and another from the transfer matrix approach given in \Cref{A transfer matrix approach R2 pp}. Further, it is shown that the results obtained from both the approach are consistent. Apart from this, such perturbations lead to interesting consequences related to parameter sequences, corresponding measure and quadrature rules.\par 

This paper is organized as follows: In \Cref{Co-polynomials of R2 type}, theory related to the study of new polynomials arising out of the old ones is developed. \Cref{Connection with unit circle} involves study of two forms of perturbation in the chain sequence $\{\lambda_{n+1}\}_{n \geq 1}$, one is the co-dilation and another called complementary chain sequences. Their effects on the corresponding OPUC are analysed by the means of separate illustrations. In \Cref{Distribution of Zeros R2 pp}, some interlacing, monotonicity and sharing properties concerning zeros are discussed. The key ingredient of this section are CRR polynomials. In \Cref{A transfer matrix approach R2 pp}, the structural relation based on a transfer matrix approach presented in \cite{Castillo co-polynomials on real line 2015} are found when finite number of perturbations are made. 

\section{Co-polynomials of $R_{II}$ type}\label{Co-polynomials of R2 type}
Consider the recurrence relation
\begin{align}\label{special R2}
	&\mathcal{P}_{n+1}(x) = (x-c_n)\mathcal{P}_n(x)-\lambda_n (x^2+1)\mathcal{P}_{n-1}(x), \quad n\geq 0, \\    
	& \nonumber	\mathcal{P}_{-1}(x) = 0, \qquad \mathcal{P}_0(x) = 1,
\end{align}
where $\{c_n\}_{n \geq 0}$ is a real sequence and $\{\lambda_n\}_{n \geq 1}$ is a positive chain sequence. Perturbing the coefficients $c_n$ and $\lambda_n$ at any fixed level $k$ according to Favard theorem \cite[Theorem 3.5]{Esmail masson JAT 1995} generates new sequence of $R_{II}$ polynomials. Modification of $c_n$, called generalised co-recursive, modification of $\lambda_n$, called generalised co-dilated and modification of both $c_n$ and $\lambda_n$ at the same level $k$, called generalised co-modified, are considered.\par 
For orthogonal polynomials satisfying \eqref{OPRL TTRR R2}, perturbation of the form \eqref{co-recursive condition R2} for the case $k=0$, called co-recursive was introduced and studied in \cite{chihara PAMS 1957}. Perturbations like \eqref{co-dilated condition R2} for the case $k=1$, called co-dilated were introduced by \cite{Dini thesis 1988}. The general case, called generalized co-recursive and generalized co-dilated arising from perturbation of coefficients in TTRR at any level was studied in \cite{Paco perturbed recurrence 1990}.\par 
The properties of the co-modified classical orthogonal polynomials has been studied in \cite{Ronveaux vigo 1988}. Interlacing properties and some new inequalities involving the zeros of co-modified OPRL, called co-polynomials on real line (COPRL) has been investigated in \cite{Castillo co-polynomials on real line 2015, Castillo chapter 2017}. For details on co-polynomials on unit circle and co-polynomials of $R_{I}$ type, see \cite{Castillo perturbed szego 2014} and \cite{Castillo monotonicity R1 2015, swami vinay 2021} respectively. \par 


Recall that polynomials generated by \eqref{R2} are called $R_{II}$ polynomials. Consider the case when both modifications \eqref{co-recursive condition R2} and \eqref{co-dilated condition R2} are done at the same level $k$ in \eqref{R2}. A more general situation can be the introduction of $\mu_k$ at the level $k$ and $\nu_k'$ at the level $k'$.
Following the terminology given in \cite{Paco perturbed recurrence 1990}, the recurrence relation so obtained will be called generalised co-modified polynomials of $R_{II}$ type (or simply {\it co-polynomials of $R_{II}$ type}) and are given by 
\begin{align*}
	\mathcal{P}_{n+1}(x;\mu_k,\nu_k) &= (x-c_n)\mathcal{P}_{n}(x;\mu_k,\nu_k)-\lambda_n (x^2+1)\mathcal{P}_{n-1}(x;\mu_k,\nu_k), \qquad \mbox{$n < k$}, \\
	\mathcal{P}_{k+1}(x;\mu_k,\nu_k) &= (x-c_k-\mu_k)\mathcal{P}_{k}(x;\mu_k,\nu_k)-\nu_k \lambda_k (x^2+1)\mathcal{P}_{k-1}(x;\mu_k,\nu_k), ~ \mbox{$n = k$}, \\
	\mathcal{P}_{n+1}(x;\mu_k,\nu_k) &=(x-c_n)\mathcal{P}_{n}(x;\mu_k,\nu_k)-\lambda_n (x^2+1)\mathcal{P}_{n-1}(x;\mu_k,\nu_k),\qquad \mbox{$n \geq k+1$}.
\end{align*}
\begin{remark}
It can be easily verified that associtaed polynomials of order r, $\mathcal{P}^{(r)}_{n}$ satisfy
\begin{align*}
\mathcal{P}^{(r)}_{n+1}(x) = (x-c_{n+r})\mathcal{P}^{(r)}_n(x)-\lambda_{n+r} (x^2+1)\mathcal{P}^{(r)}_{n-1}(x), \quad n\geq 0,
\end{align*}
with initial conditions $\mathcal{P}^{(r)}_{-1}(x)=0$ and $\mathcal{P}^{(r)}_{0}(x)=1$. Hence, by Favard theorem \cite{Esmail masson JAT 1995}, there exits a moment functional with respect to which $\{\mathcal{P}^{(r)}_{n} \}$ is also a sequence of $R_{II}$ poynomials.
\end{remark}
We can solve the last recurrence relation in terms of $\mathcal{P}_n$ and associated polynomials of order $r$, i.e. $\mathcal{P}^{(r)}_{n-r}$ to obtain the representation of new perturbed polynomials in terms of unperturbed ones in the following way (see \Cref{Theorem s_k_x R2 pp}). \par 
Let us define
\begin{align}\label{casarotti determinant R2}
	D(u_n,v_n)=\begin{array}{|cc|}
		u_n & v_n \\
		{u}_{n+1} & {v}_{n+1},
	\end{array}
\end{align}
the Casoratti determinant associated with two arbitrary sequences $\{u_n\}$ and $\{v_n\}$. From the theory of linear difference equations, we know that the two sequences are said to be linearly independent if the Casoratti determinant is non-zero for every $n$ \cite{milne 1951}. Let us consider
\begin{align*}
	\mathbb{P}_{n+1} &=
	\begin{bmatrix}
		\mathcal{P}_{n+1} & \mathcal{P}_n	
	\end{bmatrix}^T, \qquad
	\mathbf{T}_n= \begin{bmatrix}
		x-c_n & -\lambda_n (x^2+1) \\
		1 & 0
	\end{bmatrix},
\end{align*}
Now, from \eqref{R2}, we have
\begin{align*}
	\mathbb{P}_{n+1}&= \mathbf{T}_n \mathbb{P}_n = \begin{bmatrix}
		x-c_n & -\lambda_n (x^2+1) \\
		1 & 0
	\end{bmatrix} \begin{bmatrix}
		\mathcal{P}_n \\
		{\mathcal{P}}_{n-1}
	\end{bmatrix}, 
\end{align*}
\begin{align}\label{P_n+1 to P_0 R2 pp}
	\mathbb{P}_{n+1} &= (\mathbf{T}_n \ldots \mathbf{T}_0) \mathbb{P}_{0}, \qquad
	\mathbb{P}_{0} =
	\begin{bmatrix}
		\mathcal{P}_{0} & \mathcal{P}_{-1}	
	\end{bmatrix}^T.
\end{align}
Notice that $\mathcal{P}^{(k)}_{n-k}$ is a solution of the recurrence relation \eqref{R2} with initial conditions $\mathcal{P}^{(k)}_{-1}=0$ and $\mathcal{P}^{(k)}_{0}=1$. It is easy to verify that 
\begin{align*}
	\begin{bmatrix}
		\mathcal{P}_{n+1}	& \mathcal{P}^{(k)}_{n-k+1}\\
		\mathcal{P}_{n}	& \mathcal{P}^{(k)}_{n-k}
	\end{bmatrix} &=\mathbf{T}_n \begin{bmatrix}
		\mathcal{P}_{n}	& \mathcal{P}^{(k)}_{n-k}\\
		\mathcal{P}_{n-1}	& \mathcal{P}^{(k)}_{n-k-1}
	\end{bmatrix}.
\end{align*}
Hence,
\begin{align}\label{D of P_n P_{n-k} R2 pp}
	D(\mathcal{P}_{n},\mathcal{P}^{(k)}_{n-k})=\lambda_n (x^2+1)D(\mathcal{P}_{n-1},\mathcal{P}^{(k)}_{n-k-1}).
\end{align}
Let X denote the set of zeros of $\mathcal{P}_{k-1}$. From \eqref{D of P_n P_{n-k} R2 pp}, we get
\begin{align}\label{D_Pn R2 pp}
	D(\mathcal{P}_{n},\mathcal{P}^{(k)}_{n-k})&= \prod_{j=k}^{n}\lambda_j (x^2+1)^{n-k}\mathcal{P}_{k-1}.
\end{align}
which means that $\mathcal{P}_{n}$ and $\mathcal{P}^{(k)}_{n-k}$ are linearly independent in $\mathbb{R}\backslash X$. 
\begin{theorem}\label{Theorem s_k_x R2 pp}
	For $x \in \mathbb{R}\backslash X$, the following relations hold:
	\begin{align*}
		\mathcal{P}_{n}(x;\mu_k,\nu_k)&=\mathcal{P}_{n}(x), \qquad \mbox{$n \leq k$}, \\
		{\mathcal{P}}_{n}(x;\mu_k,\nu_k) &= \mathcal{P}_{n}(x) - \mathcal{S}_k(x)\mathcal{P}^{(k)}_{n-k}(x), \qquad \mbox{$n > k$},
	\end{align*}
	where $\mathcal{S}_k(x)=\mu_k \mathcal{P}_k(x)+(\nu_k-1)\lambda_k (x^2+1) \mathcal{P}_{k-1}(x)$.
\end{theorem}
\begin{proof}
	It follows from the theory of difference equations that any solution of \eqref{special R2} will be linear combination of two linearly independent solutions and vice-versa, thus assuming $\mathcal{P}_{k-1}(x) \neq 0$, there exists functions $A(x)$ and $B(x)$, such that
	\begin{align*}
		{\mathcal{P}}_{n}(x;\mu_k,\nu_k) = A(x)\mathcal{P}_{n}(x)+B(x)\mathcal{P}^{(k)}_{n-k}(x).
	\end{align*}
	For $n=k$ and $n=k+1$, we have $A(x)=1$ and $B(x)=-\mu_k \mathcal{P}_k(x)+(1-\nu_k)\lambda_k (x^2+1) \mathcal{P}_{k-1}(x)$ which proves the theorem.
\end{proof}
\begin{remark}
Observe that if $\mu_k=0$, degree of $\mathcal{S}_k(x)$ is $k+1$, also, if $\nu_k=1$, degree of $\mathcal{S}_k(x)$ is $k$. From here, we conclude that for $\mu_k \neq 0$ and $\nu_k \neq 1$, $\mathcal{S}_k(x)$ is a polynomial of degree $k+1$. Note that this is not the case with OPRL satisfying \eqref{OPRL TTRR R2}, see \cite[Theorem 2.1]{Castillo co-polynomials on real line 2015} {\rm(also see \cite{Paco perturbed recurrence 1990})}. Moreover, this result is different from the one obtained for $R_I$ polynomials in \cite[Theorem 3.1]{swami vinay 2021}.
\end{remark}

\begin{theorem}
If $\mathcal{P}_{n}(x;\mu_k,\nu_k)$ and $\mathcal{P}_{n}(x)$ have common zeros, then such zero will also be common to $\mathcal{S}_k(x)$.
\end{theorem}
\begin{proof}
Suppose $\beta$ is a common zero of $\mathcal{P}_{n}(x;\mu_k, \nu_k)$ and $\mathcal{P}_{n}(x)$ such that  $\mathcal{S}_{k}(\beta) \neq 0$. Let $Y:=\{y \in \mathbb{R} : \mathcal{S}_k(y)=0 \} $. Then, since $\beta \in \mathbb{R}\backslash (X \cup Y)$, Theorem \ref{Theorem s_k_x R2 pp} implies $\mathcal{P}^{(k)}_{n-k}(\beta)=0$, a violation to linear independence of $\mathcal{P}_{n}(x)$ and $\mathcal{P}^{(k)}_{n-k}(x)$.
\end{proof}

\begin{corollary}\label{corollary zeros sharing R2 pp}
The co-recursive polynomials $\mathcal{P}_{n}(x;\mu_k)$ and $\mathcal{P}_{n}(x)$ have $k$ zeros in common which are same as the zeros of $\mathcal{P}_{k}(x)$.
\end{corollary}

The recurrence relations for the generalised co-recursive $R_{II}$ polynomials $\mathcal{P}_{n+1}(z;\mu_k)$ and for the generalised co-dilated $R_{II}$ polynomials $\mathcal{P}_{n+1}(z;\nu_k)$ and the representation of new perturbed polynomials in terms of unperturbed ones can be obtained by substituting $\nu_k=1$ and $\mu_k=0$ respectively in \Cref{Theorem s_k_x R2 pp}.
\begin{remark} 
Polynomials obtained after perturbations in \eqref{OPRL TTRR R2}, in \eqref{Szego recurrence R2} and in \eqref{R1 R2 pp}  have been called co-polynomials on real line (COPRL) \cite{Castillo co-polynomials on real line 2015}, co-polynomials on unit circle (COPUC) \cite{Castillo perturbed szego 2014} and co--polynomials of $R_{I}$ type respectively \cite{swami vinay 2021}. Following an analogous nomenclature, we may call perturbed polynomials introduced \Cref{Co-polynomials of R2 type} as \textbf{co-polynomials of $R_{II}$ type}.
\end{remark}

\begin{example}\label{example 1 R2 paper}
In the recurrence relation \eqref{R2}, following \cite{Bracciali Pereira ranga 2020}, allow the parameters to be constant sequences, e.g. let $c_n=0$, $\lambda_{n}=\dfrac{1}{4}$, $a_n=-i$ and $b_n=i$. We have
\begin{align}
&\mathcal{P}_{n+1}(x) = x\mathcal{P}_n(x)-\dfrac{1}{4} (x^2+1)\mathcal{P}_{n-1}(x), \quad n \geq 1, \label{Example 1 R2 recurrence} \\
&	\mathcal{P}_{0}(x) = 1, \qquad \mathcal{P}_1(x) = x. \nonumber
\end{align}
\end{example}
It can be verified that
\begin{align*}
\mathcal{P}_{n}(x)= i\left(\dfrac{x-i}{2}\right)^{n+1}-i\left(\dfrac{x+i}{2}\right)^{n+1}, \quad n \geq 0.
\end{align*}

Assuming $\mu$ to be positive, if we make a perturbation at the beginning of the sequence
$\mathcal{P}_1(x) = x-\mu$. Consequently, we have
\begin{align*}
	\mathcal{P}_{n+1}(x;\mu) = \mathcal{P}_{n+1}(x)-\mu\mathcal{P}_{n}(x).
\end{align*}
This is the case for perturbation at $k=0$ level. Therefore, $\mathcal{P}_{0}(x)=1$ and $\mathcal{P}^{(0)}_{n}(x)=\mathcal{P}_{n}(x)$ along with \Cref{Theorem s_k_x R2 pp} yields the above relation. If $\mu = 1$, then
\begin{align}
\mathcal{P}_{n+1}(x;1) &= \mathcal{P}_{n+1}(x)-\mathcal{P}_{n}(x) \nonumber \\
&= i(x-2-i)\left(\dfrac{x-i}{2}\right)^{n+1}-i(x-2+i)\left(\dfrac{x+i}{2}\right)^{n+1}.
\end{align} 
The polynomial $\mathcal{P}_{n+1}(x;1)$ satisfies \eqref{Example 1 R2 recurrence} with perturbed initial conditions $\mathcal{P}_{0}(x;1) = 1$, $\mathcal{P}_1(x;1) = x-1$, and that verifies our results in this section.
	
\section{Connection with unit circle: Complementary Chain sequences and Co-dilation}\label{Connection with unit circle}
\subsection{Co-polynomials and consequences on unit circle}\label{Co-polynomials and consequences on unit circle}
Following \cite{Esmail Ranga 2018}, consider a special form of $R_{II}$ recurrence 
\begin{align}\label{special R2 shifted}
	&\mathcal{P}_{n+1}(x) = (x-c_{n+1})\mathcal{P}_n(x)-\lambda_{n+1} (x^2+1)\mathcal{P}_{n-1}(x), \quad n\geq 1, \\    
	& \nonumber	\mathcal{P}_{0}(x) = 1, \qquad \mathcal{P}_1(x) = x-c_1,
\end{align}
where $\{c_n\}_{n \geq 1}$ is a real sequence and $\{\lambda_{n+1}\}_{n \geq 1}$ is a positive chain sequence. In a recent work, such recurrence relation is shown to be related to a generalized eigenvalue problem whose eigenvalues are zeros of $\mathcal{P}_n(x)$ \cite[Theorem 1.1]{Esmail Ranga 2018}. The transformation
\begin{align*}
	\xi(x) =\dfrac{x+i}{x-i} ,
\end{align*}
maps real line onto the unit circle with a slit $\mathbb{T}= \{\xi=e^{i\theta}, 0<\theta<2\pi\}$. The inverse of this transformation is $x(\xi)= i\dfrac{\xi+1}{\xi-1}$. With this transform, it was pointed out in \cite{Esmail Ranga 2018} that $R_{II}$ polynomials are related to sequence of polynomials $\{r_n\}$ given by 
\begin{align}\label{r_n to P_n}
	r_n(\xi)= \dfrac{2^n\mathcal{P}_n(x)}{(x-i)^n}, \quad n \geq 1,
\end{align}
satisfying 
\begin{align*}
	r_{n+1}(\xi) = ((1+ic_{n+1})\xi+(1-ic_{n+1}))r_{n}(\xi)-4\lambda_{n+1}\xi r_{n-1}(\xi), \quad n \geq 1,
\end{align*}
that turn out to be related to OPUC via the relation
\begin{align}\label{phi_n to r_n}
	\phi_{n-1}(\xi)= \dfrac{r_n(\xi)-2(1-l_n)r_{n-1}(\xi)}{(z-1)\prod_{j=1}^{n}(1+ic_j)}, \quad n \geq 1,
\end{align}
or equivalently,
\begin{align}\label{phi_n to P_n}
	\phi_{n-1}(\xi)= \dfrac{-i2^{n-1}}{\prod_{j=1}^{n}(1+ic_j)}\dfrac{1}{(x-i)^{n-1}}[\mathcal{P}_n(x)-(1-l_n)(x-i)\mathcal{P}_{n-1}(x)], \quad n \geq 1,
\end{align}
by the modification of the corresponding measure of orthogonality \cite[Theorem 1.2]{Esmail Ranga 2018}.
Paraorthogonal polynomials on the unit circle are given by $\rho_n(z)=z\phi_{n-1}(z)- \tau_n\phi^*_{n-1}(z)$ where $|\tau_n|=1$ and $\phi^*_{n}(z)=z^n\overline{\phi_n(\bar{z})}$ and their zeros lie on $\partial\mathbb{D}$ \cite{Esmail book, Simon part1 2005}. The sequence of $R_{II}$ polynomials $\{\mathcal{P}_n(x)\}_{n \geq 0}$ satisfying \eqref{special R2 shifted} are shown to be related to a certain sequence $\{z\phi_{n-1}(z)-\tau_n\phi^*_{n-1}(z)\}$ of para-orthogonal
polynomials on the unit circle \cite{Esmail Ranga 2018}.

\begin{theorem}\cite[Theorem 3.2]{Esmail Ranga 2018}\label{Esmail ranga R2 pp theorem 3.2}
	Consider the recurrence relation \eqref{special R2 shifted} where $\{\lambda_{n+1}\}_{n \geq 1}$ is a positive chain sequence having minimal parameter sequence $\{\l_{n+1}\}_{n \geq 0}$ and let $\mu$ be the positive measure on the unit circle such that its Verblunsky coefficients are
	\begin{align}\label{alpha_n and tau_n}
		\alpha_{n-1}=-\dfrac{1}{\tau_n}\dfrac{1-2l_{n+1}-ic_{n+1}}{1-ic_{n+1}}, \qquad \tau_{n}=\tau_{n-1} \dfrac{1-ic_n}{1+ic_n}, \quad n\geq 1,
	\end{align}
	where $\tau_0=1$.
\end{theorem}
The effect on the OPUC when the recurrence coefficients of OPRL \eqref{OPRL TTRR R2} are perturbed has been studied in \cite{Castillo co-polynomials on real line 2015, Castillo perturbed polynomials via szego transform 2017}. The reverse situation, i.e. effect on OPRL by modifying the Verblunsky coefficients has been the matter of research in \cite{Castillo perturbed szego 2014}. In the same, some connection formulas for recurrence coefficients have also been derived explicitly. On the same lines, the next theorem gives relation among Verblunsky coefficients of the perturbed OPUC and POPUC, recurrence coefficients of \eqref{special R2 shifted}, \eqref{Special R2 perturbed} and Verblunsky coefficients of OPUC and POPUC associated with $R_{II}$ type recurrence.
\begin{theorem}\label{gamma_n to alpha_n}
	Let $\{\gamma_n \}_{n \geq 1}$ and $\{\eta_n\}_{n \geq 1}$ be the sequence of Verblunsky coefficients of the corresponding COPUC and Co-POPUC associated to 
	\begin{align}
		&\mathcal{P}_{n+1}(x) = (x-a_{n+1})\mathcal{P}_n(x)-b_{n+1} (x^2+1)\mathcal{P}_{n-1}(x), \quad n\geq 0, \label{Special R2 perturbed}\\
		a_{n+1} = &\begin{cases}
			c_{n+1}, & n \neq k \\
			a_{k+1}, & n =k
		\end{cases}, \qquad b_{n+1} = \begin{cases}
			\lambda_{n+1}, & n \neq k \\
			b_{k+1}, & n =k
		\end{cases} , \label{new a_n and b_n}
	\end{align}
	Then for $n \geq k$,
	\begin{align}\label{gamma_n-1 formula in theorem}
& \gamma_{n-1}=\dfrac{1-ia_{n}}{1+ia_{n}}\dfrac{1+ic_{n}}{1-ic_{n}}\dfrac{1-ic_{n+1}}{1-ia_{n+1}} \left[\alpha_{n-1} -\dfrac{1}{\tau_n} \left \{\dfrac{2(l_{n+1}-l'_{n+1})+i(c_{n+1}-a_{n+1})}{1-ic_{n+1}}\right \} \right], 
	\end{align}
	where $\{l'_{n+1}\}_{n \geq 0}$ is the minimal parameter sequence of $\{b_{n+1}\}_{n \geq 1}$ and 
\begin{align}
\begin{split}\label{eta_n cases}
\eta_n = \tau_n,~ n\leq k, \quad &\eta_{k+1}=\dfrac{1-ia_{k+1}}{1+ia_{k+1}}\dfrac{1+ic_{k+1}}{1-ic_{k+1}}\tau_{k+1},~ n=k+1, \\ &\eta_{n}=\eta_{n-1} \dfrac{1-ic_n}{1+ic_n},~ n>k+1.
\end{split}
\end{align} 

\end{theorem}
\begin{proof}
With $\{l'_{n+1}\}_{n \geq 0}$ as the minimal parameter sequence of the positive chain sequence $\{b_{n+1}\}_{n \geq 1}$ and let $\{\gamma_n \}_{n \geq 1}$ be the Verblunsky coefficients, then corresponding to \eqref{Special R2 perturbed} from \Cref{Esmail ranga R2 pp theorem 3.2}, we have
\begin{align}\label{gamma_n-1 eta_n in proof}
\gamma_{n-1}= -\dfrac{1}{\eta_n}\dfrac{1-2l'_{n+1}-ia_{n+1}}{1-ia_{n+1}}, \qquad \eta_{n}=\eta_{n-1} \dfrac{1-ia_n}{1+ia_n}, \quad n\geq 1.	
\end{align}
Now, \eqref{eta_n cases} is straight forward from expression for $\tau_n$ in \eqref{alpha_n and tau_n} and $\eta_n$ in \eqref{gamma_n-1 eta_n in proof}. Now, using \eqref{eta_n cases} and expression for $\alpha_{n-1}$ in \eqref{alpha_n and tau_n} in $\gamma_{n-1}$ defined above gives \eqref{gamma_n-1 formula in theorem} after some elementary computations.
\end{proof}

\begin{corollary}\label{CCS corollary}
If \eqref{Special R2 perturbed} is considered with co-dilation only and further with $c_{n+1}=a_{n+1}=0$, $\forall~ n$. Then \Cref{gamma_n to alpha_n} implies
\begin{align*}
\gamma_{n-1}= \alpha_{n-1}-\hat{\alpha}_{n-1},~ \mbox{where} ~\hat{\alpha}_{n-1} =2(l_{n+1}-l'_{n+1}).	
\end{align*}
The term $\hat{\alpha}_{n-1}$ quantifies the change in the Verblunsky coefficients caused by co-dilation.
\end{corollary}
\subsection{Co-dilation and Chain sequences}\label{Co-dilation and Chain sequences}
Considering $c_{n}=0$, $n\geq 1$, $\lambda_{2}=1/2$ and $\lambda_{n+1}=1/4$, $n\geq 2$ in \eqref{special R2 shifted}. The sequence $\{\lambda_{n+1}\}_{n \geq 1}$ is a SPPCS (single parameter positive chain sequence) and uniquely determines its parameters i.e. $l_1=0$ and $\{l_{n+1}\}_{n \geq 1} = 1/2$ which is also the minimal parameter sequence of $\{\lambda_{n+1}\}_{n \geq 1}$. With these conditions, it can be easily verified that 
\begin{align}
	\mathcal{P}_{n}(x)&= \left(\dfrac{x-i}{2}\right)^n+\left(\dfrac{x+i}{2}\right)^n, \label{P_n x for lambda= 1/2}\\
	\mathcal{Q}_{n}(x)& = i\left(\dfrac{x-i}{2}\right)^n-i\left(\dfrac{x+i}{2}\right)^n, \quad n \geq 1, \label{Q_n x for lambda= 1/2}
\end{align}
which from \eqref{P_n x for lambda= 1/2}, in the view of \eqref{r_n to P_n} and transform $z = \dfrac{x+i}{x-i}$ implies
\begin{align*}
	r_n(z)&=z^n+1, \quad n \geq 1,
\end{align*}
and hence, from \eqref{phi_n to r_n}, we have the monic OPUC as
\begin{align*}
	\phi_{n}(z)=z^n, \quad n \geq 0.
\end{align*}
The associated Verblunsky coefficients are $\alpha_n = 0$, $n \geq 0$ which can be computed via \eqref{alpha_n and tau_n} and the corresponding probability measure $\mu$ is the Lebesgue measure given by $d\mu(\xi) = \dfrac{1}{2i\pi \xi}d\xi$. \par 
Consider a new sequence $\{\lambda'_{n+1}=1/4\}_{n \geq 1}$ obtained by a single modification at the starting of the chain sequence $\{\lambda_{n+1}\}_{n \geq 1}$ i.e. $\lambda'_{2}=\nu_2\lambda_{2}$ where $\nu_2 = 1/2$ is the dilation factor chosen in such a way that $\{\lambda'_{n+1}\}_{n \geq 1}$ is again a positive chain sequence. \par 
As mentioned in the \Cref{Introduction}, we can obtain the new sequence of $R_{II}$ polynomials via direct computation. To be precise, it is interesting to note that this new chain sequence $\{\lambda'_{n+1}\}_{n \geq 1}$ is not a SPPCS. Its minimal and maximal parameter sequences are given as
\begin{align*}
	l'_{n+1}=\dfrac{n}{2n+2}, \quad n \geq 0, \quad \rm{and} \quad M'_{n+1}=\dfrac{1}{2}, \quad n \geq 0,
\end{align*}
respectively. This fact, alongwith \Cref{gamma_n to alpha_n} gives new sequence of Verblunsky coefficients $\gamma_{n-1}=-\dfrac{1}{n+1}$. Using $\{\gamma_{n-1}\}_{n \geq 1}$, from the Szeg\H{o} recurrence \eqref{Szego recurrence R2}, monic OPUC can be constructed as
\begin{align*}
	\phi'_{n}(z)=\dfrac{(n+1)z^{n}+nz^{n-1}+(n-1)z^{n-2}+\ldots+2z+1}{n+1}, \quad n \geq 0.
\end{align*}
Comparing it to \eqref{phi_n to r_n}, we get the palindromic polynomials
\begin{align*}
	r'_n(z)=\dfrac{z^{n+1}-1}{z-1}, \quad n \geq 0.
\end{align*}
which from \eqref{r_n to P_n} gives new sequence of $R_{II}$ polynomials
\begin{align*}
	\mathcal{P}'_{n}(x)= i\left(\dfrac{x-i}{2}\right)^{n+1}-i\left(\dfrac{x+i}{2}\right)^{n+1}, \quad n \geq 0.
\end{align*}
The corresponding measure of orthogonality is found to be $d\mu(\xi) = \dfrac{(1-\xi)(\xi-1)}{4i\pi \xi^2}d\xi$ \cite{Esmail Ranga 2018}. For a recent treatise on numerical quadrature arising out of these polynomials, see \cite{Bracciali Pereira ranga 2020}. \par 

\subsection{Complementary chain sequences}
Related to the chain sequences is the important concept of complementary chain sequences which can be defined as the following:
\begin{definition}\cite{KKB ranga swami 2016}
If $\{l_{n} \}_{n \geq 1}$ is the minimal parameter sequence of $\{\lambda_{n+1} \}_{n \geq 1}$, then the sequence $\{d_{n+1} \}_{n \geq 1}$, whose minimal parameter sequence is a new sequence $\{k_{n} \}_{n \geq 1}$, is the complementary chain sequence of $\{d_{n+1} \}_{n \geq 1}$ when $k_1 = 0$ and $k_n = 1-l_n$, $n \geq 2$.
\end{definition}
The motivation for this study follows from the fact that the expression \eqref{phi_n to P_n} gives an explicit relation among $R_{II}$ polynomials, Szeg\H{o} polynomials and minimal parameter sequence $\{l_{n} \}_{n \geq 1}$ of $\{\lambda_{n+1} \}_{n \geq 1}$.
\begin{theorem}\label{CCS theorem}
Let $\{c_{n} \}_{n \geq 1}$ and $\{\lambda_{n+1}\}_{n \geq 1}$ be as given in \eqref{special R2 shifted}. Suppose $\{l_{n+1}\}_{n \geq 0}$ is the minimal parameter sequence of $\{\lambda_{n+1}\}_{n \geq 1}$ and further let $\{k_{n+1}\}_{n \geq 0}$ be the minimal parameter sequence of positive chain sequence $\{d_{n+1}\}_{n \geq 1}$ which is obtained as complementary to $\{\lambda_{n+1}\}_{n \geq 1}$. Suppose
\begin{align*}
\alpha_{n-1}=-\dfrac{1}{\tau_n}\dfrac{1-2l_{n+1}-ic_{n+1}}{1-ic_{n+1}}, \qquad \beta_{n-1}=-\dfrac{1}{\tau_n}\dfrac{1-2k_{n+1}-ic_{n+1}}{1-ic_{n+1}},  
\end{align*}
where $\tau_{n}=\tau_{n-1} \dfrac{1-ic_n}{1+ic_n}$ for $n\geq 1$. Let $\mu(z)$ and $\nu(z)$ be the probability measures corresponding to the Verblunsky coefficients $\alpha_{n-1}$ and $\beta_{n-1}$. Then, the following can be derived:
{\rm (1)} If $\{\lambda_{n+1}\}_{n \geq 1}$ has multiple parameter sequences and measure $\mu(z)$ is such that the value of the integral $\displaystyle\int_{\mathbb{T}}\dfrac{1}{|z-1|^2}d\mu(z)$ is finite, then corresponding to complementary chain sequence $\{d_{n+1}\}_{n \geq 1}$, measure $\nu(z)$ is such that the value of the integral $\displaystyle\int_{\mathbb{T}}\dfrac{1}{|z-1|^2}d\nu(z)$ is infinite.\\
{\rm (2)} $\beta_{n-1}=-\bar{\alpha}_{n-1} \bar{\tau}_n \bar{\tau}_{n+1}, ~ n\geq 1$ i.e. the Verblunsky coefficients $\beta_{n-1}$ are just the linear transformation of $\alpha_{n-1}$.\\
{\rm (3)} If $c_n=(-1)^nc$, $c \in \mathbb{R}$, $\beta_{n-1}=-\dfrac{1+ic}{1-ic}\alpha_{n-1}$, $n\geq 1$.
\end{theorem}
\begin{proof}
Given that $\{\lambda_{n+1}\}_{n \geq 1}$ is a non-SPPCS, it follows from \cite[Lemma 2.4]{KKB ranga swami 2016} that its complementary chain sequence $\{d_{n+1}\}_{n \geq 1}$ is a SPPCS. By Wall's criteria
\begin{align*}
	\sum_{n=2}^{\infty}\prod_{j=2}^{n}\dfrac{k_j}{1-k_j} = \infty.
\end{align*}
Using above fact in the proof of \cite[Theorem 3.2]{Esmail Ranga 2018} proves the first part of the theorem. \par 
The second part can be proved using the following set of expressions
\begin{align*}
\beta_{n-1}&=-\dfrac{1}{\tau_n}\dfrac{1-2k_{n+1}-ic_{n+1}}{1-ic_{n+1}}= -\dfrac{1}{\tau_n}\dfrac{-1+2l_{n+1}-ic_{n+1}}{1-ic_{n+1}} \\
\text{and}~  \bar{\alpha}_{n-1}&=\dfrac{1}{\bar{\tau}_n} \dfrac{-1+2l_{n+1}-ic_{n+1}}{1+ic_{n+1}}. 
\end{align*}
The proof of third part is straight forward.
\end{proof}

\subsection{An illustration}
Consider the $R_{II}$ type recurrence
\begin{align*}
	&\mathcal{P}_{n+1}(x) = \left(x-\dfrac{\theta}{\zeta+n+1}\right)\mathcal{P}_n(x)-\dfrac{1}{4}\dfrac{n(2\zeta+n+1)}{(\zeta+n)(\zeta+n+1)} (x^2+1)\mathcal{P}_{n-1}(x), \\
\text{with} \quad &c_n = \dfrac{\theta}{\zeta+n}, \quad \lambda_{n+1} =\dfrac{1}{4}\dfrac{n(2\zeta+n+1)}{(\zeta+n)(\zeta+n+1)}, \quad \zeta, \theta \in \mathbb{R},~ \zeta > -1, ~ n\geq 1.
\end{align*}
The positive chain sequence $\{\lambda_{n+1}\}_{n \geq 1}$ has multiple parameter sequence except $-1/2 \geq \zeta > -1$. Its minimal parameter sequence is 
\begin{align*}
	l_{n+1} = \dfrac{n}{2(\zeta+n+1)}, \quad n\geq 0.
\end{align*}
As shown in \cite{Esmail Ranga 2018}, the $R_{II}$ polynomials for $\theta=0$ (which means when $c_n=0$) are given as
\begin{align*}
\mathcal{P}_{n}(x)= \dfrac{(2\zeta+2)_n}{(\zeta+1)_n}\dfrac{(x-i)^n}{2^n}F(-n,\zeta+1;2\zeta+2;\dfrac{-2i}{x-i}), \quad n\geq 1,
\end{align*}
and hence, from \eqref{phi_n to P_n}, the monic OPUC, corresponding Verblunsky coefficients and associated probability measure are given, respectively by
\begin{align}
&\phi_n(z)=\dfrac{(2\zeta+3)_n}{(\zeta+2)_n}F(-n,\zeta+2;2\zeta+3;1-z), \nonumber \\
&\alpha_{n-1} = -\dfrac{(\zeta+1)_n}{(\zeta+2)_n}, \quad n\geq 1, \qquad \mbox{and} \label{alpha_n-1}\\
& d\mu(e^{it})=\dfrac{2^{2\zeta+1}(\Gamma(\zeta+2))^2}{\pi\Gamma(2\zeta+3)}\sin^{2\zeta+2}t/2 dt. \nonumber
\end{align}
As it can be seen that polynomials $\mathcal{P}_{n+1}(x)$ depends on $\zeta$, we will use the notation $\mathcal{P}^{(\zeta)}_{n+1}(x)$ and similarly, $\lambda^{(\zeta)}_{n+1}$ for $\lambda_{n+1}$ and $c^{(\zeta)}_n$ for $c_n$. 
Now, consider the parameter sequence $k_n=1-l_n=(2\zeta+n+1)/(2\zeta+2n)$, $n\geq 1$. The resulting chain sequence is 
\begin{align*}
d^{(\zeta)}_{2}=\dfrac{2\zeta+3}{2\zeta+4}, \quad d^{(\zeta)}_{n+1}= \dfrac{1}{4}\dfrac{(n-1)(2\zeta+n+2)}{(\zeta+n)(\zeta+n+1)}, \quad n\geq 2.
\end{align*}
Let $\beta_{n-1}$ be the Verblunsky coefficients corresponding to measure $\nu(z)$. From \Cref{CCS theorem}, we get
\begin{align}\label{beta_n-1}
\beta_{n-1}=-\dfrac{1}{\tau_n}\dfrac{1-2k_{n+1}-ic_{n+1}}{1-ic_{n+1}},	
\end{align}
and from \eqref{phi_n to P_n}, the corresponding monic OPUC
\begin{align*}
\hat{\phi}_{n-1}(\xi)= \dfrac{-i2^{n-1}}{\prod_{j=1}^{n}(1+ic^{(\zeta)}_j)}\dfrac{1}{(x-i)^{n-1}}[\mathcal{\hat{P}}^{(\zeta)}_n(x)-(1-k_n)(x-i)\mathcal{\hat{P}}^{(\zeta)}_{n-1}(x)], \quad n \geq 1,
\end{align*}
where the polynomials $\mathcal{\hat{P}}_n(x)$ given by
\begin{align*}
\mathcal{\hat{P}}^{(\zeta)}_{n+1}(x) = (x-c^{(\zeta)}_{n+1})\mathcal{\hat{P}}^{(\zeta)}_n(x)-d^{(\zeta)}_{n+1} (x^2+1)\mathcal{\hat{P}}^{(\zeta)}_{n-1}(x), \quad n\geq 1,
\end{align*}
with $\mathcal{\hat{P}}^{(\zeta)}_{0}(x)=1$ and $\mathcal{\hat{P}}^{(\zeta)}_{1}(x)=x-c^{(\zeta)}_{1}$ are obtained. A characterization of the recurrence coefficients of $\mathcal{\hat{P}}^{(\zeta)}_n(x)$ in terms of the recurrence coefficients of $\mathcal{P}^{(\zeta)}_n(x)$ can be given as $d^{(\zeta)}_{n+1}=\lambda^{(\zeta+1)}_{n}$, $n\geq 2$ and $c^{(\zeta)}_{n+1}=c^{(\zeta+1)}_{n+1}$, as it is difficult to find a closed expression for $\mathcal{\hat{P}}^{(\zeta)}_n(x)$. For this case too, $c^{(\zeta)}_{n+1}(=c_{n+1})=0$, $n\geq 0$ which, by \Cref{CCS theorem}, implies that $\alpha_{n-1}=-\beta_{n-1}$, $n\geq 1$. This can also be verified using the expression for $k_n$ in \eqref{beta_n-1} and then comparing it with \eqref{alpha_n-1}. \par 


Results analogous to \Cref{gamma_n to alpha_n} and \Cref{CCS corollary} can be stated using complementary chain sequences. In this regard, we give the following remark:
\begin{remark}
If $\{k'_{n+1}\}_{n \geq 0}$ is the minimal parameter sequence of the co-dilated complementary chain sequences, say $\{d'_{n+1}\}_{n \geq 1}$. Then \Cref{CCS corollary} implies
\begin{align*}
	\gamma_{n-1}= \alpha_{n-1}+\hat{\alpha}_{n-1},~ \mbox{where} ~\hat{\alpha}_{n-1} =2(k_{n+1}-k'_{n+1}).	
\end{align*}
\end{remark}
To illustrate the consequences of co-dilation in complementary chain sequences, consider the chain sequence $\{\lambda_{n+1}\}_{n \geq 1}$ with minimal parameters $l_1=0$, $l_{n+1}=(n+2)/(2n+2)$, $n \geq 1$. By \cite[Lemma 2.5]{KKB ranga swami 2016}, $1/2<l_{n+1}<1$ implies $\{\lambda_{n+1}\}_{n \geq 1}$ is SPPCS. Now, the sequence $\{k_{n+1}\}_{n \geq 0}$ with $k_1=0$ and $k_{n+1}=1-l_{n+1}=n/(2n+2)$ is the minimal parameter sequence of the complementary chain sequence $\{d_{n+1}\}_{n \geq 1}$ where $d_{n+1}$ is the constant sequence $\{1/4\}$ which is known to be non-SPPCS. In this case, it is shown in \Cref{Co-dilation and Chain sequences} that the measure $\mu$ associated to the Szeg\H{o} polynomials is the Lebesgue measure given by $d\mu(z) = \dfrac{1}{2i\pi z}dz$ and with respect to which $\displaystyle\int_{\mathbb{T}}\dfrac{1}{|z-1|^2}d\mu(z)$ does not exist, which in turn means that $z=1$ is in the support of the measure $\mu$.\par 
The co-dilated complementary chain sequence $\{d'_{n+1}\}_{n \geq 1}$ obtained on perturbing $d_2$ as $d'_2=\nu_2d_2$ where $\nu_2=2$ is given as $d'_2=1/2$ and $d'_{n+1}=1/4$, $n \geq 2$. Its minimal parameter sequence is $\{k'_{n+1}\}_{n \geq 0}$ where $k'_1=0$ and $k'_{n+1}=1/2$ which is also maximal and this makes $\{d'_{n+1}\}_{n \geq 1}$ a SPPCS. Further, from \cite[Example 1]{Esmail Ranga 2018}, the associated measure is $d\mu'(z) = \dfrac{(1-z)(z-1)}{4i\pi z^2}dz$ with respect to which it is clear that $\displaystyle\int_{\mathbb{T}}\dfrac{1}{|z-1|^2}d\mu'(z)$ exists which implies $z=1$ is not in the support of $\mu'$.
	
\section{Interlacing and monotonicity of Zeros}\label{Distribution of Zeros R2 pp}
We need the following results given in \cite{Esmail Ranga 2018} to prove our next result.
\begin{theorem}\cite[Theorem 2.1]{Esmail Ranga 2018}\label{Ranga leading coefficient}
The polynomial $\mathcal{P}_n$ is of exact degree $n$ with positive leading coefficient.	Precisely, if we denote by $\mathfrak{p}$ the leading coefficient of $\mathcal{P}_n(x)$, then $\mathfrak{p}_0 = 1$, $\mathfrak{p}_1 = 1$ and 
\begin{align*}
	0 < (1-l_{n-1})=\dfrac{\mathfrak{p}_n}{\mathfrak{p}_{n-1}} < 1, \quad n \geq 2.
\end{align*}
Here, $\{l_{n}\}_{n \geq 0}$ is the minimal parameter sequence of the positive chain sequence $\{\lambda_{n}\}_{n \geq 1}$.
\end{theorem}
\begin{theorem}\cite[Theorem 2.2]{Esmail Ranga 2018}\label{Ranga zeros interlacing R2}
The zeros $x^{(n)}_j$, $j=1,2, \ldots n$ of $\mathcal{P}_n$ are real and simple. Assuming the ordering $x^{(n)}_j < x^{(n)}_{j-1}$ for the zeros, we also have the interlacing property
\begin{align*}
	x^{(n+1)}_{n+1} < x^{(n)}_{n}<x^{{n+1}}_n < \ldots < x^{(n+1)}_2 < x^{(n)}_{1} <x^{(n+1)}_1, \quad  n \geq 1.
\end{align*}
\end{theorem}
\begin{theorem}\label{interlacing theorem mu R2 pp}
Let $n \geq k$ and $x^{(n)}_j(\mu)$ and $x^{(n)}_j$, $j=1,2, \ldots l$ be the $l$ non common real zeros corresponding to $\mathcal{P}_{n}(x;\mu_k)$ and $\mathcal{P}_{n}(x)$. If $\mu < 0$, then 
\begin{align}\label{interlacing non common R2 pp}
	x^{(n)}_l(\mu) < x^{(n)}_l < x^{(n)}_{l-1}(\mu) < x^{(n)}_{l-1}< \ldots < x^{(n)}_{1}(\mu) < x^{(n)}_{1},
\end{align}
where the role of the zeros $x^{(n)}_j(\mu)$ and $x^{(n)}_j$, $j=1,2, \ldots l$ gets intercharged when $\mu > 0$.
\end{theorem}
\begin{proof}
	For $n \geq k$, using \eqref{casarotti determinant R2}, we can write
\begin{align*}
D(\mathcal{P}_{n}(x), \mathcal{P}_{n}(x;\mu_k)) &= \mathcal{P}_{n}(x)\mathcal{P}_{n+1}(x;\mu_k)-\mathcal{P}_{n}(x;\mu_k)\mathcal{P}_{n+1}(x) \\
&= \lambda_n(x^2+1)D(\mathcal{P}_{n-1}(x), \mathcal{P}_{n-1}(x;\mu_k)),
\end{align*}
which yeilds after a few computation
\begin{align}\label{D(P_n P_mu R2 pp)}
D(\mathcal{P}_{n}(x), \mathcal{P}_{n}(x;\mu_k)) &= -\mu_k(\lambda_n \ldots \lambda_k)(x^2+1)^{n-k}\mathcal{P}^2_{k}(x).
\end{align}
From Theorem \ref{Ranga zeros interlacing R2}, we have $(-1)^j \mathcal{P}_{n+1}(x^{(n)}_j)> 0$, $j=1,2, \ldots n$. Recall that $x^{(n)}_j$, $j=1,2, \ldots n$ are the $n$ real zeros corresponding to $\mathcal{P}_{n}(x)$. Let us assume that $\mathcal{P}_{n}(x;\mu_k)$ and $\mathcal{P}_{n}(x)$ have no common zeros. When $\mu < 0$ and $\{\lambda_{n}\}_{n \geq 1}$ is a positive chain sequence, \eqref{D(P_n P_mu R2 pp)} implies $-\mathcal{P}_{n}(x^{(n)}_j;\mu_k)\mathcal{P}_{n+1}(x^{(n)}_j) > 0$, leading to 
\begin{align*}
	(-1)^{j+1}\mathcal{P}_{n}(x^{(n)}_j;\mu_k) > 0, \quad j \geq 1.
\end{align*}
which shows that the zeros of $\mathcal{P}_{n}(x;\mu_k)$ and $\mathcal{P}_{n}(x)$ will interlace as \eqref{interlacing non common R2 pp}. For $\mu > 0$, sign in \eqref{D(P_n P_mu R2 pp)} changes and subsequently, the result follows from similar analysis.
\end{proof}
\begin{example}
To illustrate \Cref{interlacing theorem mu R2 pp} when $\mathcal{P}_{n}(x)$ and $\mathcal{P}_{n}(x;\mu_k)$ have common zeros, let us consider special $R_{II}$ type recurrence \eqref{special R2} with $c_n = 0$, $n \geq 0$ and $\lambda_n =1/4$, $n \geq 1$. The sequence $\{\lambda_n\}_{n \geq 1}$ is a positive chain sequence.  
\end{example}
Rewriting \eqref{special R2} with the above assumptions, we get
\begin{align}\label{Special R2 with c_n=0} 
	\mathcal{P}_{n+1} = x\mathcal{P}_n-\dfrac{1}{4} (x^2+1)\mathcal{P}_{n-1}, \quad n\geq 1. 
\end{align} 
Precisely, following two cases are considered depending upon the sign of $\mu_k$. All the
\begin{wraptable}{r}{8cm}
	\caption{}\label{T1 R2}
	\renewcommand{\arraystretch}{1.1}
	\centering
	\begin{tabular}{|p{3cm}|p{4cm}|}
		\hline
		Zeros of $\mathcal{P}_{9}(x)$& Zeros of $\mathcal{P}_{9}(x;-0.7)$\\
		\hline
		$-$1.376381920 & $-$1.376381920\\
		\hline
		1.376381920 & 1.376381920\\
		\hline
		0.3249196962 & 0.3249196962\\
		\hline
		$-$0.3249196962 & $-$0.3249196962\\
		\hline
		3.077683537& 1.685063442\\
		\hline
		0.7265425280 & .4309372535\\  
		\hline
		0 &  $-$.2251211415\\
		\hline
		$-$0.7265425280 & $-$1.137754967 \\
		\hline
		$-$3.077683537 & $-$10.75312459 \\
		\hline
	\end{tabular}
\end{wraptable}
 calculations are performed and figures are drawn using Maple 18 with Intel Core i3-6006U CPU @ 2.00 Ghz and 8 GB RAM.\par
\textbf{Case I:} $\mu_k < 0$. \par 
Clearly, $\mathcal{P}_{9}(x)$ and $\mathcal{P}_{9}(x;-0.7)$ have four zeros in common (see \Cref{T1 R2}) which is in accordance with \Cref{corollary zeros sharing R2 pp}. The location of zeros of $\mathcal{P}_{9}(x)$ and $\mathcal{P}_{9}(x;-0.7)$ in \Cref{Fig1} verifies \Cref{interlacing theorem mu R2 pp}.

\begin{wraptable}{r}{8cm}
	\caption{}\label{T2 R2}
	\renewcommand{\arraystretch}{1.1}
	\centering
	\begin{tabular}{|p{2.5cm}|p{4cm}|}
		\hline
		Zeros of $\mathcal{P}_{7}(x)$& Zeros of $\mathcal{P}_{7}(x;0.43)$\\
		\hline
		0 & 0 \\
		\hline
		1 & 1\\  
		\hline
		$-$1 & $-$1\\
		\hline
		2.414213562 & 3.336754639\\
		\hline
		.4142135624 & .5388256504\\
		\hline
		$-$.4142135624 & $-$.2996923982 \\
		\hline
		$-$2.414213562 & $-$1.855887891\\
		\hline
	\end{tabular}
\end{wraptable}
\textbf{Case II:} $\mu_k > 0$. \\
Observe that $\mathcal{P}_{7}(x)$ and $\mathcal{P}_{7}(x;0.43)$ have three zeros in common (\Cref{T2 R2}) which is in accordance with corollary \ref{corollary zeros sharing R2 pp}. Now, let us look (see \Cref{Fig2}) at the location of uncommon zeros of $\mathcal{P}_{7}(x)$ and $\mathcal{P}_{7}(x;0.43)$. \Cref{interlacing theorem mu R2 pp} is true for this possibility as well.

\begin{figure}[H]
	\includegraphics[scale=0.8]{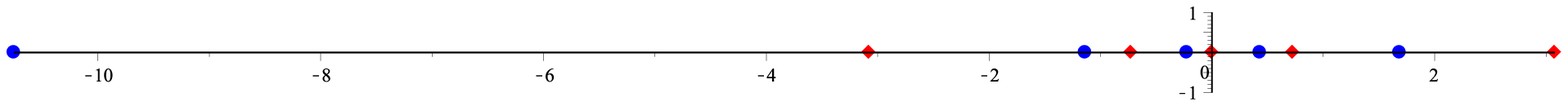}
	\caption{Zeros of $\mathcal{P}_{9}(x)$ (red diamonds) and $\mathcal{P}_{9}(x;-0.7)$ (blue circles)}
	\label{Fig1}
\end{figure}

\begin{figure}[H]
	\includegraphics[scale=0.8]{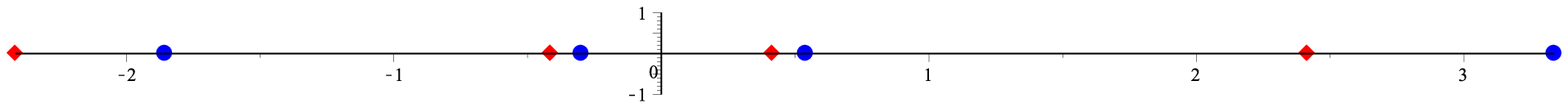}
	\caption{Zeros of $\mathcal{P}_{7}(x)$ (red diamonds) and $\mathcal{P}_{7}(x;0.43)$ (blue circles)}
	\label{Fig2}
\end{figure}

\begin{example}\label{Example 2 CRR R2 pp}
The complementary Routh-Romanovski polynomials $\mathcal{P}_{n}(e;x)$ given by the hypergeometric expression \cite{Finkelshtein Ribeiro Ranga Tyaglov CRR 2020}
\begin{align*} 
	\mathcal{P}_{n}(e;x) = \dfrac{(x-i)^n}{2^n}\dfrac{(2\zeta)_n}{(\zeta)_n}{}_2 F_1 \left(-n,e; e+\bar{e};\dfrac{-2i}{x-i}\right),
\end{align*}	
satisfy the $R_{II}$- type recurrence relation
\begin{align}\label{CRR recurrence}
	\mathcal{P}_{n+1}(e;x)=(x-c^{(e)}_{n+1})\mathcal{P}_n(e;x)-\lambda^{(e)}_{n+1} (x^2+1)\mathcal{P}_{n-1}(e;x), \quad n \geq 0,
\end{align}
with $\mathcal{P}_{0}(e;x)=1$ and $\mathcal{P}_{1}(e;x)=x-c^{(e)}_1$, where
\begin{align*}
	c^{(e)}_n = \dfrac{\theta}{\zeta+n-1}, \qquad \lambda^{(e)}_{n+1} = \dfrac{1}{4}\dfrac{n(2\zeta+n-1)}{(\zeta+n-1)(\zeta+n)}, \quad n \geq 1,
\end{align*}
for $e = \zeta+i\theta$, $\zeta > 0$. They are shown \cite{Finkelshtein Ribeiro Ranga Tyaglov PAMS 2019} to be orthogonal with respect to the weight
\begin{align*}
	\omega^{(\zeta,\theta)}(x)= \dfrac{2^{2\zeta-1}|\Gamma(b)|^2 e^{\theta \pi} (e^{-arccot(x)})^{2\theta}}{\Gamma(2\zeta-1)2\pi (1+x^2)^\zeta}.
\end{align*}
The CRR polynomials are considered to verify \Cref{interlacing theorem mu R2 pp} when $\mathcal{P}_{n}(x)$ and $\mathcal{P}_{n}(x;\mu_k)$ have no common zeros. Here, parameters $\xi=10$ and $\theta=12$ are taken to be positive. The case $\theta < 0$ provides different situation.
\end{example}
 Depending upon the sign of $\mu_k$, following two cases are considered:

\textbf{Case I:} $\mu_k < 0$. \par 
Let us perturb the recurrence coefficient $c_n$ at fouth level such that $c_3 \rightarrow c_3+\mu_3$ where $\mu_3 = -0.3$ to generate a new sequence of polynomials $\mathcal{P}_{n}(x;-0.3)$. 
Note that $\mathcal{P}_{6}(x)$ and $\mathcal{P}_{6}(x;-0.3)$  have no common zeros as shown in \Cref{T3 R2}.
Clearly, these zeros satisfy the interlacing property in accordance with \Cref{interlacing theorem mu R2 pp} (see \Cref{Fig3}).

\begin{wraptable}{r}{8cm}
	\caption{}\label{T3 R2}
	\renewcommand{\arraystretch}{1.1}
	\centering
	\begin{tabular}{|p{2.5cm}|p{4cm}|}
		\hline
		Zeros of $\mathcal{P}_{6}(x)$& Zeros of $\mathcal{P}_{6}(x;-0.3)$\\
		\hline
		0.3324095627 & 0.2430260465  \\
		\hline
		0.6295725714 & 0.5966623160 \\  
		\hline
		.9197511115 &0.8619781365 \\
		\hline
		1.273243623 &1.250348102 \\
		\hline
		1.826806110 &1.813673082 \\
		\hline
		2.724441863 & 2.533580638 \\
		\hline
	\end{tabular}
\end{wraptable}
\noindent 

\textbf{Case II:} $\mu_k > 0$. \par 
For this purpose, let us perturb the recurrence coefficient $c_n$ at fourth level such that $c_3 \rightarrow c_3+\mu_3$ where $\mu_3 = 1.2$ to generate a new sequence of polynomials $\mathcal{P}_{n}(x;1.2)$. Computed zeros $\mathcal{P}_{7}(x)$ and $\mathcal{P}_{7}(x;1.2)$ are listed in \Cref{T4 R2}. 
Note that $\mathcal{P}_{7}(x)$ and $\mathcal{P}_{7}(x;1.2)$  have no common zeros. Now, let us look (see \Cref{Fig4}) at the location of these zeros. The results follows \Cref{interlacing theorem mu R2 pp}.

\begin{figure}[H]
	\includegraphics[scale=0.8]{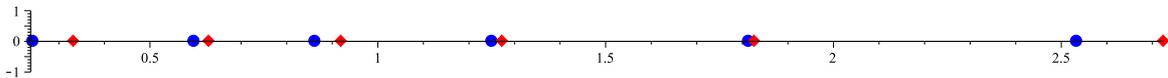}
	\caption{Zeros of $\mathcal{P}_{6}(x)$ (red diamonds) and $\mathcal{P}_{6}(x;-0.3)$ (blue circles)}
	\label{Fig3}
\end{figure}

\begin{figure}[H]
	\includegraphics[scale=0.8]{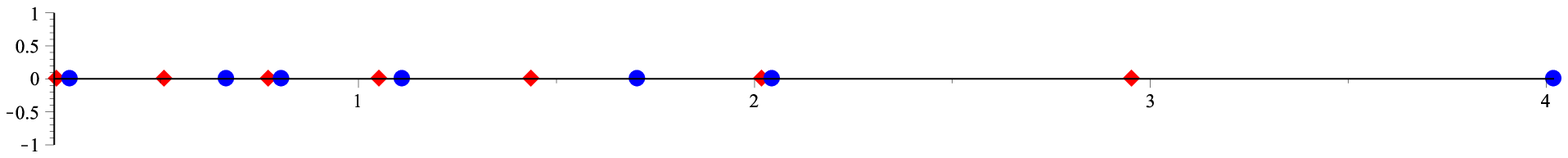}
	\caption{Zeros of $\mathcal{P}_{7}(x)$ (red diamonds) and $\mathcal{P}_{7}(x;1.2)$ (blue circles) }
	\label{Fig4}
\end{figure}
\begin{wraptable}{r}{8cm}
	\caption{}\label{T4 R2}
	\renewcommand{\arraystretch}{1.1}
	\centering
	\begin{tabular}{|p{3cm}|p{3.6cm}|}
		\hline
		Zeros of $\mathcal{P}_{7}(x)$& Zeros of $\mathcal{P}_{7}(x;1.2)$\\
		\hline
		0.2389794289 & 0.2720925666 \\
		\hline
		0.5107520351 & 0.6671359684\\  
		\hline
		0.7737829160 & 0.8062631388\\
		\hline
		1.053542176 & 1.111244050\\
		\hline
		1.437251708 & 1.705005737\\
		\hline
		2.019702291 & 2.045309617\\
		\hline
		2.953906046 & 4.019126394\\
		\hline
	\end{tabular}
\end{wraptable}
\begin{remark}
Hence, it is established that the zeros of CRR- polynomials and co-recursive CRR polynomials interlace. Moreover, it can be seen that zeros of co-recursive CRR polynomials exhibit interlacing in themselves. However, for large perturbations, some information may be lost.
\end{remark} 

\begin{corollary}
Denoting by $\mathcal{P}_{n}(x;\mu_k,\mu_{k+1})$, the polynomials obtained on perturbing two consecutive recurrence coefficients and let $x^{(n)}_j(\mu_k,\mu_{k+1})$, $j=1,2,\ldots,n$ be the corresponding zeros. With $\mu_k>0$ and $\mu_{k+1}>0$, $j$ fixed and $n \geq k$, $x^{(n)}_j(\mu_k,\mu_{k+1})$ are strictly increasing function of $\mu_k$ and $\mu_{k+1}$. Analogously, whenever $\mu_k<0$ and $\mu_{k+1}<0$, the zeros of $\mathcal{P}_{n}(x;\mu_k,\mu_{k+1})$ are strictly decreasing function of $\mu_k$ and $\mu_{k+1}$.
\end{corollary}

To illustrate the above corollary, recurrence relation satisfied by CRR polynomials is taken into consideration again.

\textbf{Case I:} $\mu_k > 0$ and $\mu_{k+1}>0$. \par 
The zeros of $\mathcal{P}_{6}(x;0.3,0.4)$ using perturbations $c_3 \rightarrow c_3+0.3$ and $c_4 \rightarrow c_4+0.4$ are listed in \Cref{T5 R2}. These zeros are represented by blue circles in \Cref{Fig5}. Now, take $\mu'_k > \mu_k$ and $\mu'_{k+1} > \mu_{k+1}$. The zeros of $\mathcal{P}_{6}(x;0.5,0.6)$ are plotted with green squares. Finally, red diamonds represent the zeros of unperturbed polynomial $\mathcal{P}_{6}(x)$. 
\begin{figure}[H]
	\includegraphics[scale=0.8]{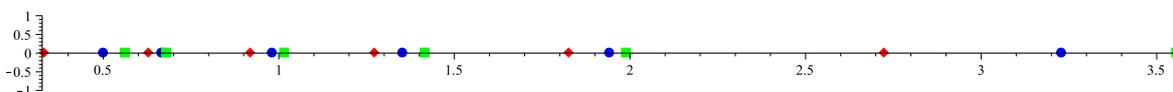}
	\caption{Zeros of $\mathcal{P}_{6}(x;0.3,0.4)$ (blue circles), $\mathcal{P}_{6}(x;0.5,0.6)$ (green squares) and $\mathcal{P}_{6}(x)$ (red diamonds).}
	\label{Fig5}
\end{figure}
\begin{table}[h]
\caption{}
\renewcommand{\arraystretch}{1.2}
	\label{T5 R2}
	\centering
	\begin{tabular}{|p{3cm}|p{4.3cm}|p{4.3cm}|}
		\hline
		Zeros of $\mathcal{P}_{6}(x)$& Zeros of $\mathcal{P}_{6}(x;0.3,0.4)$&Zeros of $\mathcal{P}_{6}(x;0.5,0.6)$\\
		\hline
		 0.3324095627 & 0.5005414531  & 0.5631868840 \\
		\hline
		 0.6295725714 & 0.6667100353  & 0.6803702451 \\  
		\hline
		0.9197511115 & 0.9815128474   & 1.016823459 \\
		\hline
		 1.273243623 & 1.352930479 & 1.416833556\\
		\hline
		1.826806110 & 1.942194994 & 1.990311187\\
		\hline
		2.724441863 & 3.229291555 & 3.556960381\\
		\hline
	\end{tabular}
\end{table}

\textbf{Case II:} $\mu_k < 0$ and $\mu_{k+1}<0$. \par
The zeros of initial CRR polynomial $\mathcal{P}_{6}(x)$ (red diamonds) and perturbed ones i.e. $\mathcal{P}_{6}(x;\break -0.2,-0.3)$ (blue circles) and $\mathcal{P}_{6}(x;-0.4,-0.5)$ (green squares) are given in \Cref{T6 R2} and plotted in \Cref{Fig6}. Observe that the result holds in this case as well.
\begin{table}[h]
	\caption{}
	\renewcommand{\arraystretch}{1.2}
	\label{T6 R2}
	\centering
	\begin{tabular}{|p{3cm}|p{4.8cm}|p{4.8cm}|}
		\hline
		Zeros of $\mathcal{P}_{6}(x)$& Zeros of $\mathcal{P}_{6}(x;-0.2,-0.3)$&Zeros of $\mathcal{P}_{6}(x;-0.4,-0.5)$\\
		 \hline
		 0.3324095627 & .1599408957  & 0.001082138805 \\
		 \hline
		 0.6295725714 & 0.5840010301  & 0.5193524908 \\  
		 \hline
		 0.9197511115 & .8799756728   & 0.8456227563 \\
		 \hline
		 1.273243623 & 1.232311882 & 1.207028405\\
		 \hline
		 1.826806110 & 1.700534290 & 1.583284261\\
		 \hline
		 2.724441863 & 2.458156724 & 2.307246095\\
		\hline
	\end{tabular}
\end{table}
\begin{figure}[H]
	\includegraphics[scale=0.8]{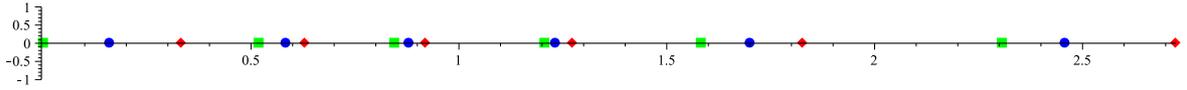}
	\caption{Zeros of $\mathcal{P}_{6}(x;-0.2,-0.3)$ (blue circles), $\mathcal{P}_{6}(x;-0.4,-0.5)$ (green squares) and $\mathcal{P}_{6}(x)$ (red diamonds).}
	\label{Fig6}
\end{figure}
\noindent \textbf{Some important observations:}  Based on the numerical experiments with several examples of $R_{II}$ polynomials given in the literature, the following has been observed:\\
{\rm(1)} It is shown that zeros of the co-recursive $R_{II}$ polynomials and their unperturbed ones have nice interlacing and monotonicity properties. However, these may not hold in the case of co-dilated and co-modified polynomials.
\begin{example}
	For the purpose of illustration, we would recall the recurrence relation \eqref{Example 1 R2 recurrence} discussed in \Cref{example 1 R2 paper}.
\end{example}

\begin{wraptable}{r}{8.5cm}
	\caption{}\label{T7 R2}
	\renewcommand{\arraystretch}{1.1}
	\centering
	\begin{tabular}{|p{3cm}|p{4cm}|}
		\hline
		Zeros of $\mathcal{P}_{9}(x)$& Zeros of $\mathcal{P}_{9}(x;0,0.6)$\\
		\hline
		0 & 0 \\
		\hline
		 $-$3.077683537 & $-$2.428062818 \\  
		\hline
		 $-$1.376381920 & $-$1.248215157 \\
		\hline
		 $-$0.7265425280&$-$.7432443609\\
		\hline
		$-$0.3249196962 &$-$.2950948155 \\
		\hline
		0.3249196962& .2950948155 \\
		\hline
		0.7265425280 & .7432443609 \\
		\hline
		1.376381920 & 1.248215157 \\
		\hline
		3.077683537 & 2.428062818\\
		\hline
	\end{tabular}
\end{wraptable}
The zeros of $\mathcal{P}_{9}(x)$ and $\mathcal{P}_{9}(x;0,0.6)$ are listed in \Cref{T7 R2} when a perturbation
$\lambda_{3} \rightarrow 0.6 \times \lambda_{3} $ is made. Clearly, the zeros do not show any kind of interlacing for this case (see \Cref{Fig7}). Similarly, one can check that co-modified $R_{II}$ polynomials too do not have such relation with the unperturbed ones.\par 

\noindent \textbf{Note:} This observation raises a question that what conditions to be imposed on $\mu_k$'s and $\nu_k$'s such that the co-dilated (or co-modified) $R_{II}$ polynomials do show some interlacing with the initial ones. Although, this aspect is not fully discussed in this manuscript and is still open.

\begin{figure}[!htb]
	\includegraphics[scale=0.8]{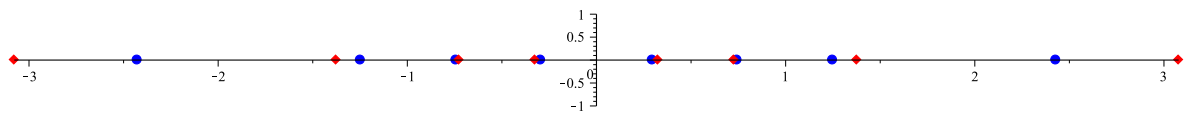}
	\caption{Zeros of $\mathcal{P}_{9}(x)$ (red diamonds) and $\mathcal{P}_{9}(x; \nu_3=0.6)$ (blue circles).}
	\label{Fig7}
\end{figure}

\noindent{\rm(2)} For the co-dilated case, it is found that when $0 < \nu_k < 1$, the zeros of $\mathcal{P}_{n}(x;\nu_{k})$, $\mathcal{P}_{n}(x)$, $\mathcal{P}_{k}(x)$ and $\mathcal{P}_{k-1}(x)$ are interlaced in a special manner i.e., whenever there are two consecutive zeros of $\mathcal{P}_{n}(x;\nu_{k})$ (see \Cref{Fig7}), there lies at least one zero of $\mathcal{P}_{k}(x)$ between them and further between two consecutive zeros of $\mathcal{P}_{n}(x)$, there is a zero of $\mathcal{P}_{k-1}(x)$ (see \Cref{Fig8,Fig9,Fig10}). However, the converse of the above fact need not true, i.e., the zeros of $\mathcal{P}_{k}(x)$ not necessarily be present only between consecutive zeros of $\mathcal{P}_{n}(x;\nu_{k})$ (see \Cref{Fig9}). This fact is numerically verified for different values of $\nu_k$ and at various levels of perturbation in recurrence relation \eqref{Example 1 R2 recurrence}, recurrence relation satisfied by CRR polynomials (\Cref{Example 2 CRR R2 pp}) and the Chebyshev polynomials of $R_{II}$ type. The Chebyshev polynomials of $R_{II}$ type \cite{Esmail masson JAT 1995} satisfy
\begin{align*}
\mathcal{P}_{n+1}(z) = (z-\sqrt{ab})\mathcal{P}_n(z)-\frac{1}{4}(z-a)(z-b)\mathcal{P}_{n-1}(z), \quad n\geq 0, ~ a,b>0.
\end{align*}
\begin{figure}[!htb]
\includegraphics[scale=0.8]{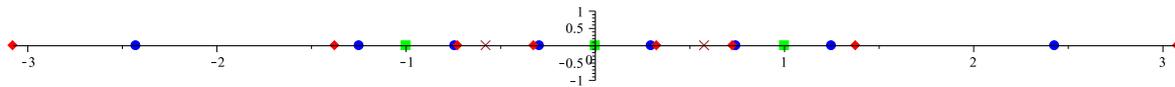}
\caption{Zeros of $\mathcal{P}_{9}(x)$ (red diamonds) obtained from \eqref{Example 1 R2 recurrence}, $\mathcal{P}_{9}(x;\nu_3=0.6)$ (blue circles), $\mathcal{P}_{3}(x)$ (green squares) and $\mathcal{P}_{2}(x)$ (brown cross).}
\label{Fig8}
\end{figure}

\begin{figure}[!htb]
	\includegraphics[scale=0.8]{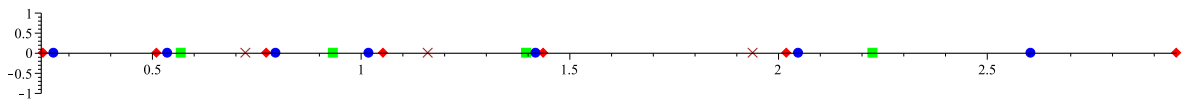}
	\caption{Zeros of CRR polynomials: $\mathcal{P}_{7}(x)$ (red diamonds), $\mathcal{P}_{7}(x;\nu_4=0.5)$ (blue circles), $\mathcal{P}_{4}(x)$ (green squares) and $\mathcal{P}_{3}(x)$ (brown cross).}
	\label{Fig9}
\end{figure}

\begin{figure}[!htb]
	\includegraphics[scale=0.8]{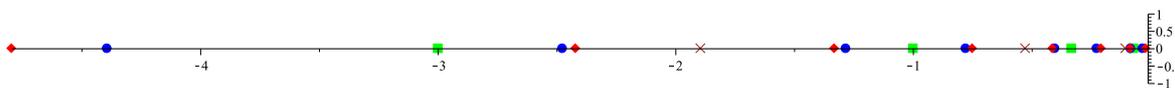}
	\caption{Zeros of Chebyshev polynomials of $R_{II}$ type with parameters $a=b=1$: $\mathcal{P}_{10}(x)$ (red diamonds), $\mathcal{P}_{10}(x;\nu_5=0.7)$ (blue circles), $\mathcal{P}_{5}(x)$ (green squares) and $\mathcal{P}_{4}(x)$ (brown cross). Note that largest zero of each polynomial is not represented in figure above for sake of clarity.}
	\label{Fig10}
\end{figure}

\noindent{\rm(3)} Let $y_{k,j}$, $j=1,2,\ldots,k+1$ be the zeros of $\mathcal{S}_{k}(x)$ defined in \Cref{Theorem s_k_x R2 pp}. With $\mu_k \neq 0$, $\nu_k \neq 1$ and $c=\frac{\nu_k-1}{\mu_k}$, it attracts our interest to look at interlacing between zeros of $\mathcal{S}_{k}(x)$, $\mathcal{P}_{k}(x)$ and $\mathcal{P}_{k-1}(x)$. If $x_{k,j}$, $j=1,2,\ldots,k$ and $x_{k-1,j}$, $j=1,2,\ldots,k-1$ are the zeros of $\mathcal{P}_{k}(x)$ and $\mathcal{P}_{k-1}(x)$ respectively arranged in increasing order, then, for $c > 0$ on $\mathbb{R}\backslash [-\infty,x_{k,1}]$, we have
\begin{align*}
x_{k-1,1}<y_{k,3}<x_{k,2}< \ldots < x_{k-1,k-1}<y_{k,k+1}<x_{k,k}.
\end{align*}
\begin{figure}[H]
	\includegraphics[scale=0.8]{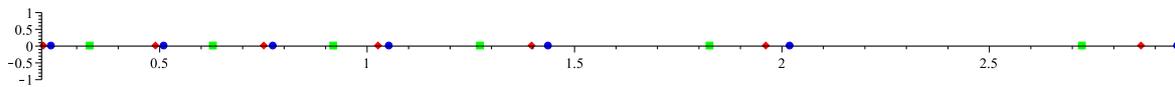}
	\caption{Zeros of $\mathcal{S}_{7}(x)$ (red diamonds) when $\mu_7=0.4$ and $\nu_7=1.2$, the CRR polynomials $\mathcal{P}_{7}(x)$ (blue circles) and $\mathcal{P}_{6}(x)$ (green squares).}
	\label{Fig11}
\end{figure}
Furthermore, when $c < 0$, then on $\mathbb{R}\backslash [x_{k,k},\infty]$, we have
\begin{align*}
	x_{k,1}<y_{k,1}<x_{k-1,1}< \ldots < x_{k,k-1}<y_{k,k-1}<x_{k-1,k-1}.
\end{align*}
\begin{figure}[H]
	\includegraphics[scale=0.8]{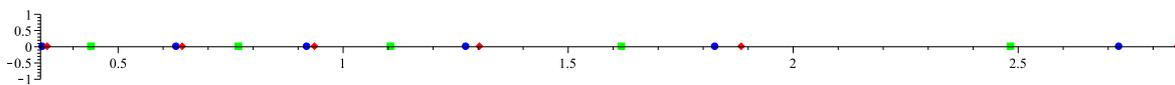}
	\caption{Zeros of $\mathcal{S}_{6}(x)$ (red diamonds) when $\mu_6=0.6$ and $\nu_6=0.8$, the CRR polynomials $\mathcal{P}_{6}(x)$ (blue circles) and $\mathcal{P}_{5}(x)$ (green squares).}
	\label{Fig12}
\end{figure}
It is possible that for some values of $\mu_k$ and $\nu_k$, $\mathcal{S}_{k}(x)$ may have complex zeros. We have avoided this situation while making above conclusions. Note that a zero of $\mathcal{S}_{7}(x)$ lying far on left side of $x_{k,1}$ (\Cref{Fig11}) and a zero of $\mathcal{S}_{6}(x)$ lying far on right side of $x_{k,k}$ (\Cref{Fig12} are omitted for pictorial clarity.

\section{A transfer matrix approach for finite perturbations}\label{A transfer matrix approach R2 pp}
Theorem \ref{Theorem s_k_x R2 pp} has some restrictions. At first, the structural relation is not useful if one is interested in finite composition of perturbations from computational point of view. Further, the structural relation doesn't hold for the entire real line. The motivation of this section is to use a transfer matrix approach to tackle these restrictions. Using this approach, the co-polynomials on the real \cite{Castillo co-polynomials on real line 2015}, perturbed Szeg\"o recurrence \cite{Castillo perturbed szego 2014}, $g$-fraction with missing terms \cite{KKB CMFT 2018} and perturbed $R_I$ recurrence \cite{swami vinay 2021} has been studied.\par

Using matrix notation \eqref{P_n+1 to P_0 R2 pp}, we have
\begin{align}\label{P_n+1 mu_k nu_k R2 pp}
\mathbb{P}_{n+1}(x;\mu_k,\nu_k) &= (\mathbf{T}_n \ldots \mathbf{T}_{k+1})\mathbf{T}_k(\mu_k,\nu_k)(\mathbf{T}_{k-1} \ldots \mathbf{T}_0) \mathbb{P}_{0},
\end{align}
where 
\begin{align*}
\mathbf{T}_k (\mu_k,\nu_k) & = \begin{bmatrix}
	x-c_k-\mu_k & -\nu_k \lambda_k (x^2+1) \\
	1 & 0	
\end{bmatrix},
\end{align*}
Using $\eqref{P_n+1 to P_0 R2 pp}$ and $\eqref{P_n+1 mu_k nu_k R2 pp}$, we deduce that following relation holds in $\mathbb{R}$. 
\begin{align}\label{P_n+1_mu_nu to P_n+1 R2 pp}
\mathbb{P}_{n+1}(x;\mu_k,\nu_k)=(\mathbf{T}_n \ldots \mathbf{T}_{k+1})\mathbf{T}_k(\mu_k,\nu_k)\mathbf{T}^{-1}_k(\mathbf{T}_n \ldots \mathbf{T}_{k+1})^{-1}\mathbb{P}_{n+1}(x),
\end{align}
The expression \eqref{P_n+1_mu_nu to P_n+1 R2 pp} has a computational edge as compared to \Cref{Theorem s_k_x R2 pp} and it holds in $\mathbb{R}$. Further, \eqref{P_n+1_mu_nu to P_n+1 R2 pp} can be improved using sequence of associated polynomials.\par
Note that the so called associated polynomials of second kind $\{\mathcal{V}_n\}$ too satisfy recurrence relation \ref{R2} with initial conditions $\mathcal{V}_0 = 0$ and $\mathcal{V}_1 = 1$. Note that $\mathcal{Q}_n = (1-\lambda_{n})^{-1}\mathcal{V}_n $ is a monic polynomial of degree $n-1$ \Cref{Ranga leading coefficient}. In the sequel, we follow the analysis given in \cite{Castillo co-polynomials on real line 2015}. \par
\begin{theorem}\label{transfer matrix theorem R2 pp}
The following relation hold in $\mathbb{R}$:
\begin{align*}
\prod_{j=1}^{k}\lambda_j (x^2+1)^k \begin{bmatrix}
	\mathcal{P}_{n+1}(x;\mu_k,\nu_k)	\\
	-\mathcal{Q}_{n+1}(x;\mu_k,\nu_k)
\end{bmatrix} & = \mathbf{N}_k \begin{bmatrix}
	\mathcal{P}_{n+1}(x)	\\
	-\mathcal{Q}_{n+1}(x)
\end{bmatrix},
\end{align*}
where $\mathbf{N}_k$ is
\begin{align*}
\mathbf{N}_k = \begin{bmatrix}
\displaystyle\prod_{j=1}^{k}\lambda_j (x^2+1)^k+\mathcal{S}_k\mathcal{Q}_k & \mathcal{S}_k\mathcal{P}_k	\\
\mathcal{Q}_k\hat{\mathcal{S}}_k & \hat{\mathcal{S}}_k \mathcal{P}_k +\displaystyle\prod_{j=1}^{k}\lambda_j (x^2+1)^k
\end{bmatrix},
\end{align*}
with 
\begin{align*}
\hat{\mathcal{S}}_k(x) &= -\mu_ k \mathcal{Q}_k(x)-(\nu_k-1)\lambda_k(x^2+1)\mathcal{Q}_{k-1}(x), \\
\mathcal{S}_k(x) &=\mu_k \mathcal{P}_k(x)+(\nu_k-1)\lambda_k (x^2+1) \mathcal{P}_{k-1}(x).
\end{align*}

\end{theorem}
\begin{proof}
Let us introduce the matrix $\mathbb{F}_{n+1}$ given by
\begin{align*}
\mathbb{F}_{n+1} & =	\begin{bmatrix}
	\mathcal{P}_{n+1}	& -\mathcal{Q}_{n+1}\\
	\mathcal{P}_{n}	& -\mathcal{Q}_{n}
\end{bmatrix} = \mathbf{T}_n \mathbb{F}_{n}. 
\end{align*}
\begin{align*}
	D(\mathcal{P}_{n+1}, -\mathcal{Q}_{n+1}) & = \lambda_n (x^2+1) det \mathbb{F}_{n} = \prod_{j=1}^{n}\lambda_j (x^2+1)^n. 
\end{align*}
then, $\mathbb{F}_{n+1} $ is non-singular. \par
Clearly, $\mathbb{F}_{n+1} $ can be written as the product of the transfer matrices
\begin{align}\label{F_n+1 to T_0}
	&	\mathbb{F}_{n+1} =  \mathbf{T}_n \mathbb{F}_{n} = \mathbf{T}_n \ldots \mathbf{T}_{k+1}\mathbf{T}_k \mathbf{T}_{k-1} \ldots \mathbf{T}_0.
\end{align}
Let $\mathbb{F}_{n+1}(\mu_k,\nu_k)$ be the polynomial matrices corresponding to co-modified polynomials of $R_{II}$ type and $\mathbb{F}_{n-k}^{(k)}$ be the product of transfer matrix for associated polynomials of order $k$. Recall that associated polynomials of order $k$ are of degree $n-k$.
\begin{align*}
	\mathbb{F}_{n-k}^{(k)} &= \mathbf{T}_n \ldots \mathbf{T}_k,
\end{align*}
then, from \eqref{F_n+1 to T_0}, we have
\begin{align}
	\mathbb{F}_{n-(k+1)}^{(k+1)} &= \mathbf{T}_n \ldots \mathbf{T}_{k+1} \nonumber  \\	
	& = \mathbf{T}_n \ldots \mathbf{T}_{k+1}\mathbf{T}_k \mathbf{T}_{k-1} \ldots \mathbf{T}_0 (\mathbf{T}_k\mathbb{F}_{k} )^{-1} \nonumber  \\
	& = \mathbb{F}_{n+1} (\mathbf{T}_k\mathbb{F}_{k} )^{-1}. \label{F_n-k+1}\\
	\mathbb{F}_{n+1}(\mu_k,\nu_k) &= \mathbf{T}_n \ldots \mathbf{T}_{k+1}\mathbf{T}_k (\mu_k,\nu_k) \mathbf{T}_{k-1} \ldots \mathbf{T}_0 \nonumber   \\
	&= \mathbb{F}_{n-(k+1)}^{(k+1)} \mathbf{T}_k (\mu_k,\nu_k) \mathbb{F}_{k}.  \label{F_n+1 mu nu}
\end{align}
Using \eqref{F_n-k+1} and \eqref{F_n+1 mu nu}, we get
\begin{align}
\mathbb{F}^T_{n+1}(\mu_k,\nu_k) &=[\mathbf{T}_k (\mu_k,\nu_k) \mathbb{F}_{k}]^T [\mathbb{F}_{n+1} (\mathbf{T}_k\mathbb{F}_{k} )^{-1}]^T \nonumber\\
&=[\mathbf{T}_k (\mu_k,\nu_k) \mathbb{F}_{k}]^T (\mathbf{T}_k\mathbb{F}_{k} )^{-T} \mathbb{F}^T_{n+1}, \label{F_n+1 transpose} 
\end{align}
which gives the following
\begin{align}
\mathbf{T}_k (\mu_k,\nu_k) \mathbb{F}_{k} &= \mathbf{T}_k (\mu_k,\nu_k) \mathbf{T}_{k-1} \ldots \mathbf{T}_0 = \mathbb{F}_{k+1}(\mu_k,\nu_k), \nonumber
\end{align}
which means
\begin{align}
\mathbb{F}^T_{k+1}(\mu_k,\nu_k) &= \begin{bmatrix}
	\mathcal{P}_{k+1}(\mu_k,\nu_k)	& \mathcal{P}_{k}\\
	-\mathcal{Q}_{k+1}(\mu_k,\nu_k)	& -\mathcal{Q}_{k}
\end{bmatrix} = [\mathbf{T}_k (\mu_k,\nu_k) \mathbb{F}_{k}]^T. \label{F_k+1 transpose}
\end{align}
Now
\begin{align}
\mathbf{T}_k\mathbb{F}_{k} &= \mathbb{F}_{k+1} = \begin{bmatrix}
	\mathcal{P}_{k+1}	& -\mathcal{Q}_{k+1}\\
	\mathcal{P}_{k}	& -\mathcal{Q}_{k}
\end{bmatrix}, \nonumber
\end{align}
and hence, by determinant formula
\begin{align}
det(\mathbb{F}_{k+1}) &= \prod_{j=1}^{k}\lambda_j (x^2+1)^k= \mathfrak{K}(x), \quad\text{(say)} \nonumber
\end{align}
implies
\begin{align}
(\mathbf{T}_k\mathbb{F}_{k})^{-T} &= \frac{1}{\mathfrak{K}(x)} \begin{bmatrix}
	-\mathcal{Q}_{k}	& -\mathcal{P}_{k} \\
	\mathcal{Q}_{k+1}	& \mathcal{P}_{k+1}
\end{bmatrix}. \label{T_k F_k tranpose inverse}
\end{align}
Using \eqref{F_k+1 transpose} and \eqref{T_k F_k tranpose inverse}, we get
\begin{align}
[\mathbf{T}_k (\mu_k,\nu_k) \mathbb{F}_{k}]^T (\mathbf{T}_k\mathbb{F}_{k} )^{-T} &=  \frac{1}{\mathfrak{K}(x)} \begin{bmatrix}
	\mathcal{P}_{k+1}(\mu_k,\nu_k)	& \mathcal{P}_{k}\\
	-\mathcal{Q}_{k+1}(\mu_k,\nu_k)	& -\mathcal{Q}_{k}
\end{bmatrix} \begin{bmatrix}
	-\mathcal{Q}_{k}	& -\mathcal{P}_{k} \\
	\mathcal{Q}_{k+1}	& \mathcal{P}_{k+1}
\end{bmatrix} \nonumber\\
&= \begin{bmatrix}\label{product matrix R2 pp}
	-\mathcal{P}_{k+1}(\mu_k,\nu_k) \mathcal{Q}_{k}+\mathcal{P}_{k}\mathcal{Q}_{k+1}	& -\mathcal{P}_{k+1}(\mu_k,\nu_k)\mathcal{P}_{k}+\mathcal{P}_{k}\mathcal{P}_{k+1} \\
	\mathcal{Q}_{k+1}(\mu_k,\nu_k)\mathcal{Q}_{k}-\mathcal{Q}_{k}\mathcal{Q}_{k+1}	& \mathcal{Q}_{k+1}(\mu_k,\nu_k)\mathcal{P}_{k}-\mathcal{Q}_{k}\mathcal{P}_{k+1}.
\end{bmatrix}
\end{align}
Now, each entry of the above matrix can be computed one by one as:
\begin{align}
	-\mathcal{P}_{k+1}(\mu_k,\nu_k) \mathcal{Q}_{k}+\mathcal{P}_{k}\mathcal{Q}_{k+1} &= [-(x-c_k-\mu_k)\mathcal{P}_{k} +\nu_k \lambda_k (x^2+1)\mathcal{P}_{k-1}]\mathcal{Q}_{k}+\mathcal{P}_{k}\mathcal{Q}_{k+1} \nonumber\\
	&= -[\mathcal{P}_{k+1}-\mathcal{S}_{k}]\mathcal{Q}_{k}+\mathcal{P}_{k}\mathcal{Q}_{k+1} \nonumber\\
	&= \mathcal{S}_{k}\mathcal{P}_{k}+\mathcal{P}_{k}\mathcal{Q}_{k+1}-\mathcal{Q}_{k}\mathcal{P}_{k+1} \nonumber\\
	&= \mathcal{S}_{k}\mathcal{Q}_{k} + \mathfrak{K}(x),  \nonumber\\
	-\mathcal{P}_{k+1}(\mu_k,\nu_k)\mathcal{P}_{k}+\mathcal{P}_{k}\mathcal{P}_{k+1} &= \mathcal{P}_{k}[\mathcal{P}_{k+1}-\mathcal{P}_{k+1}(\mu_k,\nu_k)] = \mathcal{S}_{k}\mathcal{P}_{k}, \nonumber \\
	\mathcal{Q}_{k+1}(\mu_k,\nu_k)\mathcal{Q}_{k}-\mathcal{Q}_{k}\mathcal{Q}_{k+1} &= \mathcal{Q}_{k} \hat{\mathcal{S}}_{k}, \quad \text{where} \quad \hat{\mathcal{S}}_{k} = -\mu_k \mathcal{Q}_{k}-\lambda_k (\nu_k-1)(x^2+1)\mathcal{Q}_{k-1}, \nonumber\\
	\mathcal{Q}_{k+1}(\mu_k,\nu_k)\mathcal{P}_{k}-\mathcal{Q}_{k}\mathcal{P}_{k+1}  &= [(x-c_k-\mu_k)\mathcal{Q}_{k} -\nu_k \lambda_k (x^2+1)\mathcal{Q}_{k-1}]\mathcal{P}_{k}+\mathcal{Q}_{k}\mathcal{P}_{k+1}\nonumber\\
	& = [\mathcal{Q}_{k+1}+\hat{\mathcal{S}}_{k}]\mathcal{P}_{k}-\mathcal{Q}_{k}\mathcal{P}_{k+1}  = \hat{\mathcal{S}}_{k}\mathcal{P}_{k} +\mathfrak{K}(x). \nonumber
\end{align}
Substituting above four relations reduces \eqref{product matrix R2 pp} to
\begin{align}\label{N_k upon kappa}
	[\mathbf{T}_k (\mu_k,\nu_k) \mathbb{F}_{k}]^T (\mathbf{T}_k\mathbb{F}_{k} )^{-T} &= \frac{\mathbf{N}_k}{\mathfrak{K}(x)} = \frac{\mathbf{N}_k}{\prod_{j=1}^{k}\lambda_j (x^2+1)^k}.
\end{align}
Using \eqref{N_k upon kappa} in \eqref{F_n+1 transpose} gives
\begin{align}\label{N_k into F_n+1}
	\prod_{j=1}^{k}\lambda_j (x^2+1)^k \mathbb{F}^T_{n+1}(\mu_k,\nu_k) &= \mathbf{N}_k \mathbb{F}_{n+1}.
\end{align}
After some elementary calculation, the theorem follows from \eqref{N_k into F_n+1}.
\end{proof}
\begin{corollary}
	The following relation holds for $k,m$ being non-negative integers and $m < k$:
\begin{align*}
	\prod_{j=m+1}^{k}\lambda_j (x^2+1)^{k-m} \begin{bmatrix}
		\mathcal{P}_{n+1}(x;\mu_k,\nu_k)	\\
		-\mathcal{Q}_{n+1}(x;\mu_k,\nu_k)
	\end{bmatrix} & = \mathbf{N}_k \mathbf{N}_m^{-1}\begin{bmatrix}
		\mathcal{P}_{n+1}(x;\mu_m,\nu_m)	\\
		-\mathcal{Q}_{n+1}(x;\mu_m,\nu_m).
	\end{bmatrix}
\end{align*}
\end{corollary}
\begin{proof}
The finite product \eqref{N_k into F_n+1} implies
\begin{align}
\prod_{j=1}^{k}\lambda_j (x^2+1)^k \mathbb{F}^T_{n+1}(\mu_k,\nu_k) &= \mathbf{N}_k \mathbb{F}_{n+1}, \label{1 R2}\\ \mbox{and} \quad \prod_{j=1}^{m}\lambda_j (x^2+1)^m \mathbb{F}^T_{n+1}(\mu_m,\nu_m) &= \mathbf{N}_m \mathbb{F}_{n+1}. \label{2 R2}
\end{align}
Substituting the value of $\mathbb{F}_{n+1}$ from \eqref{2 R2} in \eqref{1 R2}, we get
\begin{align*}
\prod_{j=1}^{k}\lambda_j (x^2+1)^k \mathbb{F}^T_{n+1}(\mu_k,\nu_k) &= \prod_{j=1}^{m}\lambda_j (x^2+1)^m \mathbf{N}_k \mathbf{N}_m^{-1} \mathbb{F}^T_{n+1}(\mu_m,\nu_m), 
\end{align*}
and	hence the result. 
\end{proof}
The next theorem tell us about the finite composition of perturbations.
\begin{theorem}
Let $k,m$ be two fixed non-negative integer numbers with $m < k$. Then for $n > m$, the following relation holds:
\begin{align*}
	\prod_{j=m}^{k}\prod_{l=0}^{j}\lambda_l (x^2+1)^{l} \begin{bmatrix}
		\mathcal{P}_{n+1}(x;;\mu_m,\nu_m,\ldots \mu_k,\nu_k)	\\
		-\mathcal{Q}_{n+1}(x;;\mu_m,\nu_m, \ldots \mu_k,\nu_k)
	\end{bmatrix} & = \prod_{j=m}^{k}\mathbf{N}_j \begin{bmatrix}
		\mathcal{P}_{n+1}(x)	\\
		-\mathcal{Q}_{n+1}(x)
	\end{bmatrix}.
\end{align*}
\end{theorem}
\begin{proof}
Note that $\mathbf{N}_k$ depends on the first $k+1$ original recurrence coefficients and the perturbed $c_k$ and $\lambda_k$, \Cref{N_k into F_n+1} implies
\begin{align*}
\prod_{j=1}^{k}\lambda_j (x^2+1)^k \mathbb{F}^T_{n+1}(\mu_k,\nu_k) &= \mathbf{N}_k \mathbb{F}^T_{n+1}.
\end{align*}
From \eqref{F_n+1 mu nu}, it is easy to show that
\begin{align}\label{3 R2}
\mathbb{F}_{n+1}(\mu_k,\nu_k)[\mathbf{T}_{k-1}\mathbb{F}_{k-1}]^{-1} &= \mathbf{T}_n \ldots \mathbf{T}_{k+1}\mathbf{T}_k (\mu_k,\nu_k). 
\end{align}
The expression for $\mathbb{F}_{n+1}(\mu_k,\nu_k,\mu_{k-1},\nu_{k-1})$ given by
\begin{align}\label{4 R2}
\mathbb{F}_{n+1}(\mu_k,\nu_k,\mu_{k-1},\nu_{k-1}) = \mathbf{T}_n \ldots \mathbf{T}_{k+1}\mathbf{T}_k (\mu_k,\nu_k)\mathbf{T}_{k-1}(\mu_{k-1},\nu_{k-1})\mathbf{T}_{k-2} \ldots \mathbf{T}_{0}.
\end{align}
Substituting \eqref{3 R2} in \eqref{4 R2} and taking transpose on both sides, we get
\begin{align*}
\mathbb{F}_{n+1}^T(\mu_k,\nu_k,\mu_{k-1},\nu_{k-1})=[\mathbf{T}_{k-1}(\mu_{k-1},\nu_{k-1})\mathbb{F}_{k-1}]^T[\mathbf{T}_{k-1}\mathbb{F}_{k-1}]^{-T}\mathbb{F}^T_{n+1}(\mu_k,\nu_k),
\end{align*}
which, using technique discussded in \Cref{A transfer matrix approach R2 pp}, implies
\begin{align*}
\prod_{j=1}^{k-1}\lambda_j (x^2+1)^{k-1} \mathbb{F}^T_{n+1}(\mu_k,\nu_k,\mu_{k-1},\nu_{k-1}) &= \mathbf{N}_{k-1} \mathbb{F}^T_{n+1}(\mu_k,\nu_k).
\end{align*}
In similar fashion, we have
\begin{align*}
\prod_{j=1}^{m}\lambda_j (x^2+1)^{m} \mathbb{F}^T_{n+1}(\mu_k,\nu_k \ldots \mu_{m},\nu_{m}) &= \mathbf{N}_{m} \mathbb{F}^T_{n+1}(\mu_k,\nu_k, \ldots \mu_{m-1},\nu_{m-1}).
\end{align*}
Thus, a forward substitution argument implies the theorem.
\end{proof}



\textbf{An illustration via transfer matrix approach:} In this part, apart from $\mathcal{P}_{n}(x)$ \eqref{P_n x for lambda= 1/2}, as an extra advantage, we can obtain some additional information about the second kind $R_{II}$ polynomials $\mathcal{Q}_{n}(x)$ \eqref{Q_n x for lambda= 1/2}. The transfer matrix $\mathbf{N}_k$ corresponding to the perturbation discussed in \Cref{Co-dilation and Chain sequences} is
\begin{align*}
\mathbf{N}_k = \begin{bmatrix}
	(x^2+1)/2 & -x(x^2+1)/4 \\
	0 & (x^2+1)/2
\end{bmatrix}.	
\end{align*}
The perturbed polynomials $\mathcal{\tilde{P}}_{n}(x)$ and $\mathcal{\tilde{Q}}_{n}(x)$ are obtained using \Cref{transfer matrix theorem R2 pp} as
\begin{align*}
	 (x^2+1)/2 \begin{bmatrix}
		\mathcal{\tilde{P}}_{n}(x)	\\
		-\mathcal{\tilde{Q}}_{n}(x)
	\end{bmatrix} & = \mathbf{N}_k \begin{bmatrix}
		\mathcal{P}_{n}(x)	\\
		-\mathcal{Q}_{n}(x)
	\end{bmatrix},
\end{align*}
which implies
\begin{align*}
\mathcal{\tilde{P}}_{n}(x)&= i\left(\dfrac{x-i}{2}\right)^{n+1}-i\left(\dfrac{x+i}{2}\right)^{n+1}, \\
\mathcal{\tilde{Q}}_{n}(x)& = \mathcal{Q}_{n}(x), \quad n \geq 1,
\end{align*}
It should be noted that, as expected, the polynomials $\mathcal{\tilde{P}}_{n}(x)$ are same as the polynomials $\mathcal{P}'_{n}(x)$ obtained in \Cref{Co-dilation and Chain sequences}. Further, we see that the second kind polynomials $\mathcal{Q}_{n}(x)$ remains unaffected under the aforementioned perturbation (it is easy to verify that this is not true in general) and in this case, $\mathcal{\tilde{Q}}_{n+1}(x)=\mathcal{Q}_{n+1}(x)=\mathcal{\tilde{P}}_{n}(x)$.  

\section{Concluding remarks}
\begin{enumerate}
\item \textbf{An Application:} 
The electrostatic energy $E=E(x_1,x_2, \ldots, x_m)$ \cite{Finkelshtein Ribeiro Ranga Tyaglov CRR 2020} is given by
\begin{align*}
E = -\sum_{1 \leq j \leq i \leq m} \ln|x_i-x_j|+\dfrac{\zeta_m}{2} \sum_{j=1}^{m}[\ln|x_j-i|+\ln|x_i+i|]-\theta\sum_{j=1}^{m} \arctan(x_j).
\end{align*}
where the electrostatic field consists of $i$ and $-i$ two fixed negative charges of size $\zeta_m$, $m$ movable positive unit charges and an external energy field of arctan type. In \cite[Theorem 3.1]{Finkelshtein Ribeiro Ranga Tyaglov PAMS 2019}, it was shown that the set $\{x_1^{m}(e),x_2^{m}(e)\ldots x_m^{m}(e) \}$ that maximizes E are zeros of $\mathcal{P}_{m}(e;x)$. Perturbation of the form \eqref{co-recursive condition R2} in \eqref{CRR recurrence} gives a new set of zeros, say $\{x_1^{m}(e;\mu),x_2^{m}(e;\mu)\ldots x_m^{m}(e;\mu) \}$. It would be of interest to know whether the zeros of perturbed polynomial too maximizes E or some analogue of \cite[Theorem 3.1]{Finkelshtein Ribeiro Ranga Tyaglov PAMS 2019} would exist.
\item \textbf{Open problem:} In \Cref{Distribution of Zeros R2 pp}, behaviour of zeros has been discussed for a particular case when $a_n=-i$ and $b_{n} = i$ in recurrence relation \eqref{R2}. It would be interesting to investigate behaviour of zeros for other real and complex values of $a_n$ and $b_n$ and then study the interlacing relation between perturbed and unperturbed $R_{II}$ polynomials. This demands extension of results analogous to those discussed in \cite{Esmail Ranga 2018} to a more general set up.
\end{enumerate}

\subsection*{Acknowledgments.} This research work of the second author is supported by the MATRICS Project No. MTR/2019/000029/MS of Science and Engineering Research Board, Department of Science and Technology, New Delhi, India.

\end{document}